\documentclass[12pt]{article}
\usepackage[latin1]{inputenc}
\usepackage[british]{babel}
\usepackage{cmap}
\usepackage{lmodern}

\usepackage{amssymb, amsmath, amsthm}
\usepackage[a4paper,top=25mm,bottom=25mm,left=25mm,right=25mm]{geometry}
\usepackage{ragged2e}

\usepackage{authblk} 
\usepackage{pifont}
\usepackage{graphicx}
\usepackage[usenames,dvipsnames,svgnames,table]{xcolor}
\usepackage[figuresright]{rotating}
\usepackage{xtab} 
\usepackage{longtable} 
\usepackage{multirow}
\usepackage{footnote}
\usepackage[stable]{footmisc}
\usepackage{chngpage} 
\usepackage{pdflscape} 
\usepackage{tocbibind} 

\usepackage{pgfplots}
\pgfplotsset{compat=1.18}
\pgfplotsset{every tick label/.append style={font=\footnotesize}}
\usepackage{setspace}

\makesavenoteenv{tabular}
\usepackage{tabularx}
\usepackage{booktabs}
\usepackage{threeparttable} 
\usepackage[referable]{threeparttablex} 
\newcolumntype{R}{>{\raggedleft\arraybackslash}X}
\newcolumntype{L}{>{\raggedright\arraybackslash}X}
\newcolumntype{C}{>{\centering\arraybackslash}X}
\newcolumntype{A}{>{\columncolor{gray!25}}C}
\newcolumntype{a}{>{\columncolor{gray!25}}c}

\newlength{\tablen}

\usepackage{dcolumn} 
\newcolumntype{.}{D{.}{.}{-1}}

\usepackage{tikz}
\usetikzlibrary{arrows, calc, matrix, patterns, positioning, trees}
\usepackage[semicolon]{natbib}
\usepackage[hyphens]{url}
\usepackage{hyperref} 
\hypersetup{
  colorlinks   = true,    		
  urlcolor     = blue,    		
  linkcolor    = blue,			
  citecolor    = ForestGreen,		  	
}
\usepackage{microtype}
\usepackage[justification=centering]{caption} 

\usepackage[labelformat=simple]{subcaption}

\DeclareCaptionLabelFormat{parenthesis}{(#2)}
\makeatletter
\renewcommand\p@subfigure{\arabic{figure}.}
\makeatother

\DeclareCaptionLabelFormat{parenthesis}{(#2)}
\makeatletter
\renewcommand\p@subtable{\arabic{table}.}
\makeatother

\def\addlegendimage{\csname pgfplots@addlegendimage\endcsname}

\usepackage{enumitem}

\setlist[itemize]{leftmargin=2.5\parindent}
\setlist[enumerate]{leftmargin=2.5\parindent}

\theoremstyle{plain}

\newtheorem{proposition}{Proposition}[section]

\theoremstyle{definition}

\newtheorem{example}{Example}

\theoremstyle{remark}

\def\keywords{\vspace{.5em} 
{\noindent \textit{Keywords}: }}

\def\AMS{\vspace{.5em} 
{\noindent \textbf{\emph{MSC} class}: }}

\def\JEL{\vspace{.5em} 
{\noindent \textbf{\emph{JEL} classification number}: }}

\title{Tournament schedules and incentives in a double round-robin tournament with four teams}
\author{
\href{https://sites.google.com/view/laszlocsato}{L\'aszl\'o Csat\'o}\thanks{~Corresponding author. Email: \emph{laszlo.csato@sztaki.hu} \newline
HUN-REN Institute for Computer Science and Control (HUN-REN SZTAKI), Laboratory on Engineering and Management Intelligence, Research Group of Operations Research and Decision Systems, Budapest, Hungary \newline
Corvinus University of Budapest (BCE), Institute of Operations and Decision Sciences, Department of Operations Research and Actuarial Sciences, Budapest, Hungary} \qquad \qquad
\href{https://math.bme.hu/~molontay/}{Roland Molontay}\thanks{~Email: \emph{molontay@math.bme.hu} \newline 
Budapest University of Technology and Economics, Department of Stochastics, Institute of Mathematics, M\H{u}egyetem rkp.~3., H-1111 Budapest, Hungary \newline
ELKH-BME Stochastics Research Group, M\H{u}egyetem rkp.~3., H-1111 Budapest, Hungary} \qquad \qquad
J\'ozsef Pint\'er\thanks{~Email: \emph{pinterj@edu.bme.hu} \newline 
Budapest University of Technology and Economics, Department of Stochastics, Institute of Mathematics, M\H{u}egyetem rkp.~3., H-1111 Budapest, Hungary \newline
ELKH-BME Stochastics Research Group, M\H{u}egyetem rkp.~3., H-1111 Budapest, Hungary}}
\date{\today}

\def\Dedication{
{\noindent
``\emph{Teams and leagues want to optimize their investments by playing a good schedule which seeks to meet various criteria. Good fixtures are important in order to maximize revenues, ensure the attractiveness of the games, and to keep the interest of both the media and the fans.}'' \citep[p.~1]{KendallKnustRibeiroUrrutia2010}
}}

\begin{document}

\newgeometry{top=25mm,bottom=25mm,left=25mm,right=25mm}
\maketitle
\thispagestyle{empty}
\Dedication

\begin{abstract}
\noindent
In a round-robin tournament, a team may lack the incentive to win if its final rank does not depend on the outcome of the matches still to be played. This paper introduces a classification scheme to determine these weakly (where one team is indifferent) or strongly (where both teams are indifferent) stakeless matches in a double round-robin contest with four teams. The probability that such matches arise can serve as a novel fairness criterion to compare and evaluate match schedules.
Our approach is illustrated by the UEFA Champions League group stage. A simulation model is built to compare the 12 valid schedules for the group matches. Some schedules are shown to be dominated by other schedules. It is found that the strongest team should play at home in the last round against one of the middle teams, depending on the preferences of the tournament organiser. Choosing an optimal sequence of matches with respect to the proposed metric can help to avoid uninteresting matches.
\end{abstract}

\keywords{OR in sports; round-robin tournament, simulation; sports scheduling; stakeless games; tournament design}

\AMS{62F07, 90-10, 90B35, 90B90}

\JEL{C44, C63, Z20}

\clearpage
\restoregeometry

\section{Introduction} \label{Sec1}

There is powerful support for the proposition that the level and structure of incentives influence the performance of the agents in sports \citep{EhrenbergBognanno1990}. Consequently, one of the most important responsibilities of sports governing bodies around the world is to set the right incentives for the contestants \citep{Szymanski2003}.

Knockout tournaments are guaranteed to be exciting until the end since the only way to win is not to lose. On the other hand, round-robin tournaments are sometimes decided before the final round(s), resulting in uninteresting games that should be avoided as much as possible because:
\begin{itemize}
\item
A team still in contention to win the championship positively affects attendance \citep{PawlowskiNalbantis2015};
\item
If at least one team is completely indifferent between winning, drawing, or even losing, then it may field second-team players and take other factors such as resting before the next match into account, which is unfair to the teams that have already played against this indifferent team \citep{ChaterArrondelGayantLaslier2021}.
\end{itemize}
The matches played in a round-robin tournament can be classified into three categories: competitive (neither team is indifferent), weakly stakeless (one team is indifferent), and strongly stakeless (both teams are indifferent).

To highlight the importance of this division, let us see two illustrations from the most prestigious European club competition in association football (henceforth football).

\begin{table}[ht!]
\begin{threeparttable}
\centering
\caption{Ranking in Group G of the 2018/19 UEFA Champions League after Matchday 5}
\label{Table1}
\rowcolors{3}{}{gray!20}
    \begin{tabularx}{\linewidth}{Cl CCC CCC >{\bfseries}C} \toprule \hiderowcolors
    Pos   & Team  & W     & D     & L     & GF    & GA    & GD    & Pts \\ \bottomrule \showrowcolors
    1     & Real Madrid CF & 4     & 0     & 1     & 12    & 2     & $+10$   & 12 \\
    2     & AS Roma & 3     & 0     & 2     & 10    & 6     & $+4$    & 9 \\
    3     & FC Viktoria Plze{\v n} & 1     & 1     & 3     & 5     & 15    & $-10$   & 4 \\
    4     & PFC CSKA Moskva & 1     & 1     & 3     & 5     & 9     & $-4$    & 4 \\ \bottomrule    
    \end{tabularx}
    
    \begin{tablenotes} \footnotesize
\item
Pos = Position; W = Won; D = Drawn; L = Lost; GF = Goals for; GA = Goals against; GD = Goal difference; Pts = Points. All teams have played four matches. 
    \end{tablenotes}
\end{threeparttable}
\end{table}

\begin{example} \label{Examp1}
\citep{Gieling2022}
Table~\ref{Table1} presents the standing of Group G in the 2018/19 season of the UEFA Champions League---a double round-robin tournament with four teams---before the last round. Real Madrid and Roma were guaranteed to be the group winner and the runner-up, respectively.

With the first place in the group already secured, Real Madrid had the luxury of fielding a fully rotated squad against CSKA Moskva. Perhaps not coincidentally, empty seats were everywhere in the Santiago Bernab\'eu Stadium when the final whistle came since Real Madrid lost a European home tie by more than two goals first in its history \citep{Bell2018}.
However, even the remarkable performance of CSKA Moskva was insufficient to finish third in the group because Roma also suffered a shocking defeat in Plze{\v n}.
\end{example}

\begin{table}[ht!]
\begin{threeparttable}
\centering
\caption{Ranking in Group C of the 2021/22 UEFA Champions League after Matchday 5}
\label{Table2}
\rowcolors{3}{}{gray!20}
    \begin{tabularx}{\linewidth}{Cl CCC CCC >{\bfseries}C} \toprule \hiderowcolors
    Pos   & Team  & W     & D     & L     & GF    & GA    & GD    & Pts \\ \bottomrule \showrowcolors
    1     & AFC Ajax & 5     & 0     & 0     & 16    & 3     & $+13$   & 15 \\
    2     & Sporting Clube de Portugal & 3     & 0     & 2     & 12    & 8     & $+4$    & 9 \\
    3     & Borussia Dortmund & 2     & 0     & 3     & 5     & 11    & $-6$    & 6 \\
    4     & Be{\c s}ikta{\c s} JK & 0     & 0     & 5     & 3     & 14    & $-11$   & 0 \\ \bottomrule    
    \end{tabularx}
    
    \begin{tablenotes} \footnotesize
\item
Pos = Position; W = Won; D = Drawn; L = Lost; GF = Goals for; GA = Goals against; GD = Goal difference; Pts = Points. All teams have played five matches.
    \end{tablenotes}
\end{threeparttable}
\end{table}

\begin{example} \label{Examp2}
Table~\ref{Table2} presents the standing of Group C in the 2021/22 Champions League with one round still to be played. If two or more teams were equal on points on completion of the group matches, their ranking was determined by higher number of points obtained in the matches played among the teams in question, followed by superior goal difference from the group matches played among the teams in question \citep[Article~17.01]{UEFA2021c}.
Since the result of Borussia Dortmund vs.\ Sporting CP (Sporting CP vs.\ Borussia Dortmund) was 1-0 (3-1), and Sporting CP had an advantage of three points over Borussia Dortmund, Sporting CP was guaranteed to be the runner-up. Furthermore, Ajax won the group and Be{\c s}ikta{\c s} was the fourth-placed team. Consequently, the outcomes of the two strongly stakeless games played in the last round did not influence the group ranking.
\end{example}

Even though there are a plethora of similar examples, the literature on scheduling round-robin tournaments traditionally does not consider the issue of incentives among the main criteria of fairness \citep{GoossensYiVanBulck2020}. The current paper aims to fill this research gap by investigating the influence of the schedule on the competitiveness of the matches played in the last round(s) of a double round-robin tournament with four teams.

Furthermore, even if stakeless games do not lead to unfair outcomes, there are convincing arguments to avoid them from a business perspective. Stakeless matches may have substantially less merit in terms of entertainment value, attendance, and TV-viewership \citep{DiMattiaKrumer2023, PawlowskiNalbantis2015}, which is certainly unfavourable for the organiser.


Our main contributions can be summarised as follows:
\begin{itemize}
\item
We introduce a novel fairness measure to compare and evaluate round-robin tournament schedules, namely, the probability that (weakly and strongly) stakeless matches arise in the final round(s);
\item
We show how (weakly and strongly) stakeless matches can be determined in a general double round-robin tournament with four contestants;
\item
We present that the UEFA Champions League group stage can be reliably simulated by identifying the teams with the pots from which they are drawn, guaranteeing the existence of a universal optimal schedule across all groups, independently of the outcome of the group draw;
\item
We found that competitiveness is maximised if the strongest team plays at home in the last round against one of the middle teams.
\end{itemize}
The paper conveys a clear message for tournament organisers: choosing the best sequence of games with respect to the proposed metric can help to avoid uninteresting matches. Even though fixing the schedule may affect the qualifying probabilities of the teams to some extent, this could be a reasonable price, especially because UEFA has recently made reforms in the seeding pots \citep{CoronaForrestTenaWiper2019, DagaevRudyak2019} and the qualification system \citep{Csato2022b} that have severely disadvantaged some clubs.

The remainder of the paper is structured as follows.
Section~\ref{Sec2} gives a concise overview of the literature. The game classification scheme is presented and applied to a double round-robin contest with four teams in Section~\ref{Sec3}. The simulation model, the scheduling options, and the results are detailed in Section~\ref{Sec4}. Finally, Section~\ref{Sec5} concludes.

\section{Related literature} \label{Sec2}

Although the topic can be connected to many research fields, we focus on three areas in the following: fairness, match importance, and scheduling.

\subsection*{Fairness in sports}

Stakeless games seem to be unfavourable from a sporting perspective. A team whose position in the Champions League group is already fixed may field weaker players and take into account other factors such as resting before the next match in its domestic championship. These games offer an opportunity to collude as demonstrated in sumo wrestling \citep{DugganLevitt2002} and football \citep{ElaadKrumerKantor2018}.

Suspicion of collusion can badly harm a tournament and even the reputation of the sports industry, independently of whether the match is actually fixed or not. Nonetheless, the history of football is full of examples of tacit coordination. The most famous match is probably the ``disgrace of Gij\'on'', where both West Germany and Austria were satisfied by the result of 1-0 that let the opposing teams advance to the second phase of the 1982 FIFA World Cup \citep[Section~3.9.1]{KendallLenten2017}. In order to prevent similar scandals, since then almost all tournaments---including the UEFA Champions League---are designed such that the games of the last round are played simultaneously. However, this rule has not fully guaranteed fairness as we have seen.

Some recent papers have attempted to determine the best schedule for the FIFA World Cup groups. \citet{Stronka2020} investigates the temptation to lose, resulting from the desire to play against a weaker opponent in the first round of the knockout stage. This danger is found to be the lowest if the strongest and the weakest competitors meet in the last (third) round.
Inspired by the format of the 2026 FIFA World Cup, \citet{Guyon2020a} quantifies the risk of collusion in groups of three teams, where the two teams playing the last game know exactly what results let them advance. The author identifies the match sequence that minimises the risk of collusion.
\citet{ChaterArrondelGayantLaslier2021} develop a general method to evaluate the probability of any situation in which the two opposing teams might not play competitively, and apply it to the current format of the FIFA World Cup (a single round-robin tournament with four teams). The scheduling of matches, in particular, the choice of teams playing each other in the last round, turns out to be a crucial factor for obtaining exciting and fair games.

The current work joins this line of research by emphasising that a good schedule is able to reduce the threat of tacit collusion.
However, the competition examined here---double round-robin with four teams---is much more complicated than a single round-robin with three \citep{Guyon2020a, ChaterArrondelGayantLaslier2021}, four \citep{Stronka2020, ChaterArrondelGayantLaslier2021}, or five \citep{ChaterArrondelGayantLaslier2021} teams. The proposed match classification scheme is also different from previous suggestions: for instance, \citet{ChaterArrondelGayantLaslier2021} do not distinguish between weakly and strongly stakeless games, which seem to be important categories in scheduling the UEFA Champions League groups according to our findings.

\subsection*{Match importance}

Measuring the importance of a match can be used not only to compare tournament designs but also for selecting games to broadcast, assigning referees, or explaining attendance \citep{GoossensBelienSpieksma2012}.
The first metric has probably been suggested in \citet{Jennett1984}, where the importance of a game is the inverse of the number of remaining games that need to be won in order to obtain the title. Furthermore, the measure equals zero if a team can no longer be the final winner. On the other hand, a match can be deemed important if either of the opponents can still win the league (or relegate) when all other teams will draw in the rest of their games \citep{AudasDobsonGoddard2002}.

Perhaps the most widely used quantification of match importance has been provided by \citet{Schilling1994}: the significance of an upcoming match for a particular club is determined by the difference in the probability of obtaining a prize if the fixture were won rather than lost. In other words, match importance reflects the strength of the relationship between the match result and a given season outcome. This approach is usually applied with a Monte Carlo simulation of the remaining games to estimate the final standing in the ranking \citep{BuraimoForrestMcHaleTena2022, Lahvicka2015, ScarfShi2008}.

There are further concepts of match importance.
\citet{GoossensBelienSpieksma2012} evaluate several formats for the Belgian football league with respect to the number of unimportant games that do not affect the final outcome (league title, relegation, qualification for European cups) at the moment they are played.
\citet{Geenens2014} identifies two drawbacks of this definition as it does not take into account the temporal position of the game in the tournament and the strength of the two playing teams. Therefore, the author suggests an entropy-related measure of decisiveness in terms of the uncertainty about the eventual winner prevailing in the tournament at the time of the game.
Inspired by this idea, \citet{CoronaHorrilloWiper2017} analyse how the identification of decisive matches depends on the statistical approach used to estimate the results of the matches.

\citet{FaellaSauro2021} call a match irrelevant if it does not influence the ultimate ranking of the teams involved and prove that a tournament always contains an irrelevant match if the schedule is predetermined and there are at least five contestants. This notion is somewhat akin to our classes of matches presented in Section~\ref{Sec1}.
Finally, \citet{GollerHeininger2023} propose an event importance measure: the difference between the contest reward probability distributions induced by the possible outcomes of a single event. It encompasses more complex or dynamic tournament designs and reward structures, which are common in the society. The authors show the association of the quantified importance of a match to in-match behaviour and the performance of the teams in seven major European football leagues.

Our main contribution to this research area resides in thoroughly analysing the role of schedule for match importance. Previous studies have only considered this as a promising direction for future study, which can be illustrated by two quotations:
``\emph{research could be done to investigate the influence of the schedule on match importance}'' \citep[p.~239]{GoossensBelienSpieksma2012}, and ``\emph{The tournament design, in particular the order the games were played, obviously influences the decisiveness of those games, so we have scrupulously followed that schedule in our study}'' \citep[p.~160]{Geenens2014}.
Furthermore, the above metrics of match importance are not able to account for the lack of incentives. For example, the entropy-based decisiveness measure of \citet{Geenens2014} contains the strength of the two playing teams---but Real Madrid fielded a fully rotated squad for the match discussed in Example~\ref{Examp1}.

\subsection*{Sports scheduling}

The Operational Research (OR) community devotes increasing attention to the design of sports tournaments \citep{Csato2021a, KendallLenten2017, LentenKendall2022, Wright2014}. One of the main challenges is choosing a schedule that is fair for all contestants both before and after the matches are played \citep{GoossensYiVanBulck2020}. The traditional issues of fairness in scheduling are the number of breaks (two consecutive home or away games), the carry-over effect (which is related to the previous game of the opponent), and the number of rest days between consecutive games. They are discussed in several survey articles \citep{GoossensSpieksma2012b, KendallKnustRibeiroUrrutia2010, RasmussenTrick2008, Ribeiro2012, RibeiroUrrutiadeWerra2023, VanBulckGoossensSchonbergerGuajardo2020}.

The referred studies usually consider the teams as nodes in graphs. However, they are strategic actors and should allocate their limited effort throughout the contest, or even across several contests. Researchers have recently begun to take similar considerations into account. \citet{KrumerMegidishSela2017a} investigate round-robin tournaments with a single prize and either three or four symmetric players. In the subgame perfect equilibrium of the contest with three players, the probability of winning is maximised for the player who competes in the first and the last rounds. This result holds independent of whether the asymmetry is weak or strong. However, the probability of winning is the highest for the player who competes in the second and third rounds if there are two prizes \citep{KrumerMegidishSela2020}. In the subgame perfect equilibrium of the contest with four players, the probability of winning is maximised for the player who competes in the first game of both rounds. These theoretical findings are reinforced by an empirical analysis, which includes the FIFA World Cups and the UEFA European Championships, as well as two Olympic wrestling events \citep{KrumerLechner2017}.

Two recent works focus on an issue that is similar to our topic.
\citet{Yi2020} develops an implementor-adversary approach to construct a robust schedule that maximizes suspense, the round when the winner of a round-robin tournament is decided. \citet{Gieling2022} aims to find the schedule having this particular property. While this binary metric of tension is clearly important for most round-robin tournaments, it does not always reflect the stakes appropriately. For instance, teams are mainly interested in obtaining the first two positions in the FIFA World Cup group stage \citep{ChaterArrondelGayantLaslier2021}.

To summarise, we add a novel fairness criterion, the probability of stakeless games, to compare and assess potential schedules for a round-robin tournament. Since the previous section has demonstrated why league organisers, teams, and fans would not like to see stakeless matches in the majority of sports, this measure can be used to evaluate any real-world sports timetable, for instance, the recently proposed and used schedule of FIFA World Cup South American Qualifiers \citep{DuranGuajardoSaure2017}.

\section{Game classification in a double round-robin tournament with four teams} \label{Sec3}

This section introduces the principle behind our classification scheme and shows how stakeless matches can be detected in any double round-robin contest with four contestants, a format that is widely used in sports competitions.

\subsection{The underlying idea} \label{Sec31}

Let us assume that each team strives to achieve its target or improve its ranking in a round-robin tournament. The target can be anything depending on the final ranking such as being the winner, being at least the runner-up, or avoiding relegation. However, it is reasonable in most leagues that the target is to achieve a higher rank since, for instance, a better position results in higher revenue \citep{BergantinosMoreno-Ternero2020a}.
Then three categories of matches can be distinguished:
\begin{itemize}[label=$\bullet$]
\item
\emph{Competitive game}:
Neither team is indifferent because they have still not achieved their targets or they can improve their ranking through a more favourable result on the field.
\item
\emph{Weakly stakeless game}:
One of the teams is completely indifferent as it has achieved its target or it cannot improve its position in the final ranking, independently of the outcomes of the matches still to be played. However, it has a positive probability that the other team fails to achieve its target or obtains a higher rank through a more favourable result on the field.
\item
\emph{Strongly stakeless game}:
Both teams are completely indifferent since their targets or positions in the final ranking are not influenced by the outcome of this particular match.
\end{itemize}

The above classification differs from the definitions given in the existing literature. For example, \citet{ChaterArrondelGayantLaslier2021} call a match stakeless if at least one team becomes indifferent between winning, drawing, or even losing by 5 goals difference with respect to qualification. But this notion does not consider the incentives of the opponent, which is an important factor for the competitiveness of the game, too.

\subsection{Determining stakeless games in a double round-robin tournament with four teams} \label{Sec33}

It is far from trivial to establish when a team is guaranteed to win a round-robin tournament \citep{CechlarovaPotpinkovaSchlotter2016, GusfieldMartel2002, KernPaulusma2004}. In practical applications, usually an integer programming model is developed to check whether a team is still competing for some outcome or not \citep{GoossensBelienSpieksma2012, GotzesHoppmann2022, RaackRaymondSchlechteWerner2014, RussellvanBeek2012}. However, in the relatively simple case of a double round-robin tournament with four teams, which consists of six matchdays, stakeless games can be identified in a straightforward way as follows.

Assume that a team can play its second match against another team only after it has played once against all other teams. This is a usual property of schedules in European football competitions \citep{GoossensSpieksma2012b}. Furthermore, a match is allowed to have three possible outcomes: win (3 points), draw (1 point), loss (0 points). The tie-breaking rule is either goal difference or head-to-head results \citep{Berker2014, Csato2023a}. Finally, the two matches in any round are played simultaneously, that is, exactly the results of the matches played in all the previous rounds are known.

\begin{proposition} \label{Prop1}
The final position of a team in the group ranking cannot be determined if it has played at most three matches.
\end{proposition}

\begin{proof}
In three matches, at most nine points can be collected. Therefore, at least two teams $A$ and $B$ exist such that the difference between their number of points is at most nine, hence, team $B$($A$) can be ranked higher than team $A$($B$) at the end of the tournament even if team $A$($B$) has nine points after three matches. This does not depend on the tie-breaking rule because both the goal difference and the head-to-head results can be better for team $B$($A$) than for team $A$($B$) since at least one match remains to be played between them.
\end{proof}

According to Proposition~\ref{Prop1}, the position of a team in the group ranking can be known after Matchday~4 at the earliest.

\begin{proposition} \label{Prop2}
If the tie-breaking rule is goal difference, a team is guaranteed to win the tournament after Matchday~4 if and only if it has at least seven points more than the runner-up.
If the tie-breaking rule is head-to-head records, a team is guaranteed to win the tournament after Matchday~4 if and only if it has at least seven points more than the runner-up or the following conditions hold:
\begin{itemize}
\item
it leads by six points over the runner-up; and
\item
it leads by at least seven points over the third-placed team; and
\item
it has played two matches against the runner-up.
\end{itemize}
\end{proposition}

\begin{proof}
After Matchday~4, each team has two remaining matches, where at most six points can be scored. Therefore, if ties are broken by goal difference, a team can still be ranked higher than another team at the end of the tournament if it has six points less but not if it has scored seven points less.

On the other hand, the six points advantage of team $A$ over the runner-up $B$ is sufficient to win the tournament if the head-to-head records of team $B$ against team $A$ cannot be better, which is guaranteed if teams $A$ and $B$ have already played two matches against each other. To prove this claim, assume for contradiction that team $A$ has an advantage of at most six points after Matchday~4 over the runner-up team $B$ with worse head-to-head results than team $B$. If team $A$ has scored only nine points, then team $B$ has scored three points against team $A$, thus, there exists a third team $C$ with at least five points due to its win against team $B$ and its two matches played against the fourth team $D$. Then team $B$ cannot be the runner-up. Consequently, team $A$ has scored at least ten points in the first four rounds and has played at most one draw besides at least three wins. Therefore, team $A$ has scored more points against team $B$ than vice versa, which is impossible.
\end{proof}

\begin{proposition} \label{Prop3}
If the tie-breaking rule is goal difference, a team will certainly be fourth in the final ranking after Matchday 4 if and only if it has at least seven points less than the third-placed team.
If the tie-breaking rule is head-to-head records, a team will certainly be fourth in the final ranking after Matchday~4 if and only if it has at least seven points less than the third-placed team or the following conditions hold:
\begin{itemize}
\item
it has six points less than the third-placed team; and
\item
it has at least seven points less than the runner-up; and
\item
it has played two matches against the third-placed team.
\end{itemize}
\end{proposition}

\begin{proof}
The proof is analogous to Proposition~\ref{Prop2}.
\end{proof}

Note that the conditions in Propositions~\ref{Prop2} and \ref{Prop3} are more complicated than the ones appearing in previous studies \citep{ChaterArrondelGayantLaslier2021, Guyon2020a} because of two reasons:
(a) a double round-robin contest contains more matches than a single round-robin contest; and
(b) the tie-breaking rule of head-to-head results increases the difficulty of the analysis compared to goal difference.

\begin{proposition} \label{Prop4}
The second- and third-placed clubs cannot be determined after Matchday 4, independently of the tie-breaking rule.
\end{proposition}

\begin{proof}
In four matches, the number of points scored is between 0 and 12.
Consequently, a team either (1) has some chance to win the tournament if it has at least six points and has not played twice against the first-placed team; or (2) can be overtaken by the fourth-placed team if it has at most six points and has not played twice against the fourth-placed team. It can be easily seen that at least one of these conditions holds for both the second-placed and the third-placed teams.
\end{proof}

After Matchday~5, it would be cumbersome to determine all possible cases to the analogy of Propositions~\ref{Prop2}--\ref{Prop4}.
However, there are only two remaining matches. Let $M$ be a high number that guarantees for the team scoring $M$ goals to be ranked above any other team having the same number of points without $M$ goals in any match. Assume that its opponent in the last round does not score any goal. Hence, it is sufficient to consider the four extreme cases where one team scores $M$ and the other scores zero goals as follows.

\begin{proposition} \label{Prop5}
Take four sets of results for the two games played on Matchday~6: (a) $M$-$0$, $M$-$0$; (b) $M$-$0$, $0$-$M$; (c) $0$-$M$, $M$-$0$; and (d) $0$-$M$, $0$-$M$; where $M$ is a high number.
The position of a team is guaranteed after Matchday~5 if it is the same in all scenarios (a) to (d).
\end{proposition}

\begin{proof}
The final ranking of team $A$ is already known at the end of the tournament if (1) there exists no team $B$ that is currently ranked higher than team $A$ but can be ranked lower than team $A$ at the end of the tournament, and (2) there exists no team $C$ that is currently ranked lower than team $A$ but can be ranked higher than team $A$ at the end of the tournament. Condition (1) holds if and only if team $A$ cannot be ranked higher than team $B$ if team $A$($B$) achieves the best (worst) possible result in its last match. Analogously, condition (2) holds if and only if team $A$ cannot be ranked lower than team $C$ if team $A$($C$) achieves the worst (best) possible result in its last match.
Since each team has only one match still to be played after Matchday~5, the required best and worst possible results are represented by two scenarios from the sets of results (a) to (d).
\end{proof}

\section{A case study: the UEFA Champions League} \label{Sec4}

Section~\ref{Sec3} has presented the classification of the games played in a double round-robin tournament with four teams into three categories. Now we want to understand how their probability depends on the schedule of the matches and which schedule should be followed to avoid weakly and strongly stakeless games to the extent possible. For this purpose, we have chosen to analyse the UEFA Champions League, probably the most prestigious annual football club competition around the world.

Since the 2003/04 season, this tournament starts with eight groups where the four teams play in a double round-robin format, that is, each team meets the other three teams in its group once home and once away. The top two clubs from each group progress to the Round of 16, where the group winners are matched with the runners-up subject to some restrictions \citep{BoczonWilson2023, KlossnerBecker2013}. Furthermore, the group winners play the second game at home, which provides an intrinsic advantage with respect to qualifying \citep{Bahamonde-BirkeBahamonde-Birke2023, Krumer2013, PagePage2007}.
The third-placed clubs go to the UEFA Europa League, the second-tier competition of European club football, while the fourth-placed clubs are eliminated.
Consequently, it can be reasonably assumed that each participating team wants to be ranked higher in the group stage.

For the draw of the Champions League group stage, a seeding procedure is followed to ensure homogeneity across groups. The 32 clubs are divided into four pots. One team is assigned from each pot to a group, subject to some restrictions: two teams from the same national association cannot play against each other, certain clashes are prohibited due to political reasons, and some clubs from the same country play on separate days when this is possible \citep{UEFA2021c}. According to the last constraint, there are pairings where one team is drawn into Groups A--D and another team is drawn into Groups E--H, hence the two teams would play on different days.

Seeding is based primarily on the UEFA club coefficients prior to the tournament. This measure of strength depends on the results achieved in the previous five seasons of the UEFA Champions League, the UEFA Europa League, and the UEFA Europa Conference League, including their qualifying phases \citep{UEFA2018g}. In order to support emerging clubs, the coefficient equals the association coefficient over the same period if it is higher than the sum of all points won in the previous five years.

Before the 2015/16 season, Pot 1 consisted of the eight strongest teams according to the coefficients, Pot 2 contained the next eight, and so on. The only exception was the titleholder, guaranteed to be in Pot 1. In the three seasons between 2015/16 and 2017/18, the reigning champion and the champions of the top seven associations were in Pot 1. The effects of this reform have been extensively analysed in the literature \citep{CoronaForrestTenaWiper2019, DagaevRudyak2019}. Since the 2018/19 season, Pot 1 contains the titleholders of both the Champions League and the Europa League, together with the champions of the top six associations. The other three pots are composed in accordance with the club coefficient ranking. The seeding rules are discussed in \citet{Csato2020a} and \citet[Chapter~2.3]{Csato2021a}. \citet{EngistMerkusSchafmeister2021} estimate the effect of seeding on tournament outcomes in European club football. 



\subsection{The simulation of match outcomes} \label{Sec41}

For any sports competition, historical data represent only a single realisation of several random factors. Hence, the analysis of tournament designs usually starts by finding a reasonable simulation technique that can generate the required number of results \citep{Csato2021b, ScarfYusofBilbao2009}.
To that end, it is necessary to connect the teams playing in the tournament studied to the teams whose performance is already known. It can be achieved by rating the teams, namely, by assigning a value to each team to measure its strength \citep{VanEetveldeLey2019}. This approach allows the identification of the teams by their ratings instead of their names.

Since there are 96 group matches in each season, the simulation model should be kept relatively simple to avoid overfitting (when a complex statistical estimation describes the random error in the data rather than the true relationships between the variables). Thus, instead of estimating attack and defence parameters, we measure the strength of the teams by UEFA club coefficients, which is a standard approach in the literature aimed at simulating the UEFA Champions League \citep{CoronaForrestTenaWiper2019, DagaevRudyak2019}. Another advantage of using UEFA club coefficients is that the decision-makers may be more willing to accept the results of a statistical method if it is based on a well-established rating.

In football, the number of goals scored is usually described by a Poisson distribution \citep{Maher1982, VanEetveldeLey2019}.
\citet{DagaevRudyak2019} propose such a model to evaluate the effects of the seeding system reform in the Champions League, introduced in 2015. Consider a single match between two clubs, and denote by $\lambda_{H,A}$ ($\lambda_{A,H}$) the expected number of goals scored by the home team $H$ (away team $A$) against the away team $A$ (the home team $H$). The probability $P_H$ of team $H$ scoring $n_{H,A}$ goals against team $A$ is given by
\begin{equation}
P_H \left( n_{H,A} \right) = \frac{\lambda_{H,A}^{n_{H,A}}\exp{\left( -\lambda_{H,A} \right)}}{n_{H,A}!},
\end{equation}
whereas the probability $P_A$ of team $A$ scoring $n_{A,H}$ goals against team $H$ is
\begin{equation}
P_A \left( n_{A,H} \right) = \frac{\lambda_{A,H}^{n_{A,H}}\exp{\left( -\lambda_{A,H} \right)}}{n_{A,H}!}.
\end{equation}

In order to determine the outcome of the match, parameters $\lambda_{H,A}$ and $\lambda_{A,H}$ need to be estimated on the basis of historical matches. We have started from the model of \citet{DagaevRudyak2019}, which relies on two empirical observations: home advantage is present in the Champions League and the number of goals scored correlates with the difference in the UEFA club coefficients. In particular:
\begin{equation}
\log \left( \lambda_{H,A} \right) = \alpha_H + \beta_H \cdot \left( R_H - \gamma_H R_A \right),
\end{equation}
\begin{equation}
\log \left( \lambda_{A,H} \right) = \alpha_A + \beta_A \cdot \left( R_A - \gamma_A R_H \right),
\end{equation}
with $R_H$ and $R_A$ being the UEFA club coefficients of the corresponding teams, and $\alpha_i, \beta_i, \gamma_i$ ($i \in \{ A,H \}$) being parameters. The main advantage of this choice is that the number of goals scored by the home and the away team depends on three different parameters.
A simpler version containing four parameters can be derived by setting $\gamma_H = \gamma_A = 1$.

We have studied two options for quantifying the strength of a club: the UEFA club coefficient and the seeding pot from which the team is drawn in the Champions League group stage. The latter can take only four different values. Furthermore, each group is guaranteed to consist of one team from each pot, consequently, the dataset contains the same number of teams for each possible value, as well as the same number of matches for any pair of ratings. For the sake of simplicity, all teams are classified according to the seeding system used in the given season.

Assuming the independence of the scores may be too restrictive because the two opposing teams compete against each other. Thus, if one team scores, then the other will exert more effort into scoring \citep{KarlisNtzoufras2003}. This correlation between the number of goals scored can be accounted for by bivariate Poisson distribution, which introduces an additional covariance parameter $c$ that reflects the connection between the scores of teams $H$ and $A$ \citep{VanEetveldeLey2019}. Thus, the expected scores of the two teams are   $\lambda_{H,A} + c$ and $\lambda_{A,H} + c$, respectively, where the correlation comes from the term $c$. This model has one more parameter.

To sum up, five model variants are considered:
\begin{itemize}
\item
6-parameter Poisson model based on UEFA club coefficients (6p coeff);
\item
4-parameter Poisson model based on UEFA club coefficients (4p coeff);
\item
6-parameter Poisson model based on pot allocation (6p pot);
\item
4-parameter Poisson model based on pot allocation (4p pot);
\item
7-parameter bivariate Poisson model based on UEFA club coefficients (Bivariate).
\end{itemize}

\begin{table}[t!]
\centering
\caption{Model parameters estimated by the maximum likelihood method \\ on the basis of Champions League seasons between 2003/04 and 2019/20}
\label{Table3}
\begin{threeparttable}
\rowcolors{3}{}{gray!20}
    \begin{tabularx}{\textwidth}{lc CCCCc} \toprule
    Parameter &       & 6p coeff & 4p coeff & 6p pot & 4p pot & 7p bivariate coeff \\ \bottomrule
    		& Estimation & 0.335 & 0.409 & 0.464 & 0.424 & 0.335 \\
    $\alpha_H$ & Lower 95\% & 0.318 & 0.403 & 0.447 & 0.419 & 0.318 \\
          	& Upper 95\% & 0.340 & 0.411 & 0.480 & 0.428 & 0.340 \\ \hline
     		& Estimation & 0.087 & 0.102 & 0.143 & 0.108 & 0.087 \\
    $\alpha_A$ & Lower 95\% & 0.072 & 0.093 & 0.123 & 0.101 & 0.072 \\
    		& Upper 95\% & 0.103 & 0.103 & 0.163 & 0.111 & 0.103 \\ \hline
    		& Estimation & 0.006 & 0.006 & $-$0.177 & $-$0.169 & 0.006 \\
    $\beta_H$ & Lower 95\% & 0.0062 & 0.0057 & $-$0.1720 & $-$0.1670 & 0.0062 \\
    		& Upper 95\% & 0.0065 & 0.0063 & $-$0.1800 & $-$0.1710 & 0.0065 \\ \hline
    		& Estimation & 0.006 & 0.006 & $-$0.182 & $-$0.175 & 0.006 \\
    $\beta_A$ & Lower 95\% & 0.0055 & 0.0057 & $-$0.1850 & $-$0.1790 & 0.0055 \\
    		& Upper 95\% & 0.0062 & 0.0063 & $-$0.1770 & $-$0.1730 & 0.0062 \\ \hline
    		& Estimation & 0.833 & --- & 0.910 & --- & 0.833 \\
    $\gamma_H$ & Lower 95\% & 0.810 & ---   & 0.884 & ---   & 0.810 \\
    		& Upper 95\% & 0.855 & ---   & 0.953 & ---   & 0.855 \\ \hline
    		& Estimation & 0.963 & --- & 0.922 & --- & 0.963 \\
    $\gamma_A$      & Lower 95\% & 0.938 & ---   & 0.897 & ---   & 0.938 \\
    		& Upper 95\% & 1.017 & ---   & 0.978 & ---   & 1.017 \\ \hline
    	& Estimation & --- & --- & --- & --- & $\exp \left( -12.458 \right)$ \\
    $c$ & Lower 95\% & --- & --- & --- & --- & $\exp \left( -13.517 \right)$ \\
    	& Upper 95\% & --- & --- & --- & --- & $\exp \left( -12.120 \right)$ \\ \toprule
	\end{tabularx}
\begin{tablenotes} \footnotesize
\item
The 95\% confidence intervals of the estimations are obtained by performing bootstrap resampling.
\end{tablenotes}
\end{threeparttable}
\end{table}

All parameters have been estimated by the maximum likelihood approach on the set of $8 \times 12 \times 17 = 1632$ matches played in the 17 seasons from 2003/04 to 2019/20. They are presented in Table~\ref{Table3}.
The optimal value of $c$, the correlation parameter of the bivariate model is positive but close to zero, hence, the bivariate Poisson model does not improve accuracy. This is in accordance with the finding of \citet{ChaterArrondelGayantLaslier2021} for the group stage of the FIFA World Cup. The reason is that the bivariate Poisson model is not able to grab a negative correlation between its components, however, the goals scored by home and away teams are slightly negatively correlated in our dataset.

The performance of the models is evaluated on two disjoint test sets, the seasons of 2020/21 and 2021/22. They are treated separately because most games in the 2020/21 edition were played behind closed doors owing to the COVID-19 pandemic, which might significantly affect home advantage \citep{BenzLopez2023, BrysonDoltonReadeSchreyerSingleton2021, FischerHaucap2021}.

Two metrics are calculated to compare the statistical models. \emph{Average hit probability} measures how accurately a model can determine the exact score of a match: we pick up the probability of the actual outcome, sum up these probabilities across all matches in the investigated dataset, and normalise this value by the number of matches and seasons. For instance, assume that two games have been played in a season such that the predicted probability for their known outcome is $0.1$ and $0.06$, respectively. The average hit probability will be $(0.1 + 0.06) / 2 = 0.08$.
A simple baseline model serves as a benchmark, where the chances are determined by relative frequencies in the seasons from 2003/04 to 2018/19.

\begin{table}[t!]
\centering
\caption{Average hit probability for the statistical models (in \%)}
\label{Table4}
\begin{threeparttable}
\rowcolors{3}{gray!20}{}
    \begin{tabularx}{0.8\textwidth}{L c CC} \toprule \hiderowcolors
    \multirow{2}[0]{*}{Model} & \multicolumn{3}{c}{Season(s)} \\
          & 2003/04--2019/20 & 2020/21 & 2021/22 \\ \bottomrule \showrowcolors   
    6p coeff & 7.016 (1) & 6.682 (2) & 6.148 (3) \\
    4p coeff & 7.012 (3) & 6.683 (1) & 6.146 (5) \\
    6p pot & 6.869 (4) & 6.443 (5) & 6.297 (1) \\
    4p pot & 6.868 (5) & 6.456 (4) & 6.297 (1) \\
    Bivariate & 7.016 (1) & 6.682 (2) & 6.148 (3) \\ \toprule
    Baseline & 6.123 (6) & 5.482 (6) & 5.485 (6) \\ \bottomrule
    \end{tabularx}
\begin{tablenotes} \footnotesize
\item
\emph{Baseline model}: The probability of any match outcome is determined by the relative frequency of this result in the training set (all seasons between 2003/04 and 2019/20).
\item
The ranks of the models are indicated in bracket.
\end{tablenotes}
\end{threeparttable}
\end{table}

The results are provided in Table~\ref{Table4}. The baseline model shows the worst performance, which is a basic criterion for the validity of the proposed methods. The bivariate Poisson variant does not outperform the 6-parameter Poisson based on UEFA club coefficients. Even though the club coefficient provides a finer measure of strength than the pot allocation, it does not result in a substantial improvement with respect to average hit probability.

The average hit probability does not count whether the prediction fails by a small margin (the forecast is 2-2 and the actual result is 1-1) or it is completely wrong (the forecast is 4-0 and the actual result is 1-3). However, there exists no straightforward ``distance'' among the possible outcomes. If the differences in the predicted and actual goals scored by the home and away teams are simply added, then the result of 2-2 will be farther from 1-1 than 2-1. But 1-1 and 2-2 are more similar than 1-1 and 2-1 from a sporting perspective since both 1-1 and 2-2 represent a draw. To resolve this issue, we have devised a distance metric for the outcome of the matches generated by the scalar product with a specific matrix, which has been inspired by the concept of Mahalanobis distance \citep{deMaesschalckJouan-RimbaudMassart2000}.

Let the final score of the game be $R_1 = (h_1, a_1)$, where $h_1$ is the number of goals for the home team, and $a_1$ is the number of goals for the away team. Analogously, denote by $R_2 = (h_2, a_2)$ the predicted result of this game. The distance between the two outcomes equals
\begin{equation} \label{eq5}
\Delta(R_1, R_2) = \sqrt{
\begin{bmatrix}h_1-h_2 &  a_1-a_2\end{bmatrix}
\begin{bmatrix}1 & -\pi \\ -\pi & 1\end{bmatrix}
\begin{bmatrix}h_1-h_2\\a_1-a_2\end{bmatrix}
},
\end{equation}
where the relative cost of adding one goal for both teams compared to adding one goal for one team is $2-2\pi$. Since it is reasonable to assume that $0 < 2-2\pi < 1$, thus, $1/2 < \pi < 1$, we consider three different values $\pi \in \left\{ 4/5;\, 9/10;\, 19/20 \right\}$.

For instance, with the final score of 2-0 and the forecast of 1-2, $h_1 - h_2 = 1$ and $a_1 - a_2 = -2$, which leads to
\begin{equation} \label{eq6}
\Delta(R_1, R_2) = \sqrt{
\begin{bmatrix} 1 &  -2 \end{bmatrix}
\begin{bmatrix} 1 + 2 \times \pi \\ -\pi - 2 \times 1 \end{bmatrix}
} = \sqrt{5 + 4 \times \pi}.
\end{equation}

\begin{table}[t!]
  \centering
  \caption{Selected values for the suggested distance of match scores in \eqref{eq5} ($\pi = 9/10$)}
  \label{Table5}
\resizebox{\textwidth}{!}{
\begin{threeparttable}
    \begin{tabular}{l cccc cccccc cccccc}
    \toprule
    Result      & 0-0   & 1-1   & 2-2   & 3-3   & 1-0   & 2-1   & 3-2   & 2-0   & 3-1   & 3-0   & 0-1   & 1-2   & 2-3   & 0-2   & 1-3   & 0-3 \\
    \bottomrule
    0-0   & \cellcolor{ForestGreen!0} 0 & \cellcolor{ForestGreen!7.45} 0.447 & \cellcolor{ForestGreen!14.9} 0.894 & \cellcolor{ForestGreen!22.3666666666667} 1.342 & \cellcolor{ForestGreen!16.6666666666667} 1 & \cellcolor{ForestGreen!19.7166666666667} 1.183 & \cellcolor{ForestGreen!24.7166666666667} 1.483 & \cellcolor{ForestGreen!33.3333333333333} 2 & \cellcolor{ForestGreen!35.75} 2.145 & \cellcolor{ForestGreen!50} 3 & \cellcolor{ForestGreen!16.6666666666667} 1 & \cellcolor{ForestGreen!19.7166666666667} 1.183 & \cellcolor{ForestGreen!24.7166666666667} 1.483 & \cellcolor{ForestGreen!33.3333333333333} 2 & \cellcolor{ForestGreen!35.75} 2.145 & \cellcolor{ForestGreen!50} 3 \\
    1-1   & \cellcolor{ForestGreen!7.45} 0.447 & \cellcolor{ForestGreen!0} 0 & \cellcolor{ForestGreen!7.45} 0.447 & \cellcolor{ForestGreen!14.9} 0.894 & \cellcolor{ForestGreen!16.6666666666667} 1 & \cellcolor{ForestGreen!16.6666666666667} 1 & \cellcolor{ForestGreen!19.7166666666667} 1.183 & \cellcolor{ForestGreen!32.4833333333333} 1.949 & \cellcolor{ForestGreen!33.3333333333333} 2 & \cellcolor{ForestGreen!48.8833333333333} 2.933 & \cellcolor{ForestGreen!16.6666666666667} 1 & \cellcolor{ForestGreen!16.6666666666667} 1 & \cellcolor{ForestGreen!19.7166666666667} 1.183 & \cellcolor{ForestGreen!32.4833333333333} 1.949 & \cellcolor{ForestGreen!33.3333333333333} 2 & \cellcolor{ForestGreen!48.8833333333333} 2.933 \\
    2-2   & \cellcolor{ForestGreen!14.9} 0.894 & \cellcolor{ForestGreen!7.45} 0.447 & \cellcolor{ForestGreen!0} 0 & \cellcolor{ForestGreen!7.45} 0.447 & \cellcolor{ForestGreen!19.7166666666667} 1.183 & \cellcolor{ForestGreen!16.6666666666667} 1 & \cellcolor{ForestGreen!16.6666666666667} 1 & \cellcolor{ForestGreen!33.3333333333333} 2 & \cellcolor{ForestGreen!32.4833333333333} 1.949 & \cellcolor{ForestGreen!48.8833333333333} 2.933 & \cellcolor{ForestGreen!19.7166666666667} 1.183 & \cellcolor{ForestGreen!16.6666666666667} 1 & \cellcolor{ForestGreen!16.6666666666667} 1 & \cellcolor{ForestGreen!33.3333333333333} 2 & \cellcolor{ForestGreen!32.4833333333333} 1.949 & \cellcolor{ForestGreen!48.8833333333333} 2.933 \\
    3-3   & \cellcolor{ForestGreen!22.3666666666667} 1.342 & \cellcolor{ForestGreen!14.9} 0.894 & \cellcolor{ForestGreen!7.45} 0.447 & \cellcolor{ForestGreen!0} 0 & \cellcolor{ForestGreen!24.7166666666667} 1.483 & \cellcolor{ForestGreen!19.7166666666667} 1.183 & \cellcolor{ForestGreen!16.6666666666667} 1 & \cellcolor{ForestGreen!35.75} 2.145 & \cellcolor{ForestGreen!33.3333333333333} 2 & \cellcolor{ForestGreen!50} 3 & \cellcolor{ForestGreen!24.7166666666667} 1.483 & \cellcolor{ForestGreen!19.7166666666667} 1.183 & \cellcolor{ForestGreen!16.6666666666667} 1 & \cellcolor{ForestGreen!35.75} 2.145 & \cellcolor{ForestGreen!33.3333333333333} 2 & \cellcolor{ForestGreen!50} 3 \\ \hline
    1-0   & \cellcolor{ForestGreen!16.6666666666667} 1 & \cellcolor{ForestGreen!16.6666666666667} 1 & \cellcolor{ForestGreen!19.7166666666667} 1.183 & \cellcolor{ForestGreen!24.7166666666667} 1.483 & \cellcolor{ForestGreen!0} 0 & \cellcolor{ForestGreen!7.45} 0.447 & \cellcolor{ForestGreen!14.9} 0.894 & \cellcolor{ForestGreen!16.6666666666667} 1 & \cellcolor{ForestGreen!19.7166666666667} 1.183 & \cellcolor{ForestGreen!33.3333333333333} 2 & \cellcolor{ForestGreen!32.4833333333333} 1.949 & \cellcolor{ForestGreen!33.3333333333333} 2 & \cellcolor{ForestGreen!35.75} 2.145 & \cellcolor{ForestGreen!48.8833333333333} 2.933 & \cellcolor{ForestGreen!50} 3 & \cellcolor{ForestGreen!65.4} 3.924 \\
    2-1   & \cellcolor{ForestGreen!19.7166666666667} 1.183 & \cellcolor{ForestGreen!16.6666666666667} 1 & \cellcolor{ForestGreen!16.6666666666667} 1 & \cellcolor{ForestGreen!19.7166666666667} 1.183 & \cellcolor{ForestGreen!7.45} 0.447 & \cellcolor{ForestGreen!0} 0 & \cellcolor{ForestGreen!7.45} 0.447 & \cellcolor{ForestGreen!16.6666666666667} 1 & \cellcolor{ForestGreen!16.6666666666667} 1 & \cellcolor{ForestGreen!32.4833333333333} 1.949 & \cellcolor{ForestGreen!33.3333333333333} 2 & \cellcolor{ForestGreen!32.4833333333333} 1.949 & \cellcolor{ForestGreen!33.3333333333333} 2 & \cellcolor{ForestGreen!48.8833333333333} 2.933 & \cellcolor{ForestGreen!48.8833333333333} 2.933 & \cellcolor{ForestGreen!64.9833333333333} 3.899 \\
    3-2   & \cellcolor{ForestGreen!24.7166666666667} 1.483 & \cellcolor{ForestGreen!19.7166666666667} 1.183 & \cellcolor{ForestGreen!16.6666666666667} 1 & \cellcolor{ForestGreen!16.6666666666667} 1 & \cellcolor{ForestGreen!14.9} 0.894 & \cellcolor{ForestGreen!7.45} 0.447 & \cellcolor{ForestGreen!0} 0 & \cellcolor{ForestGreen!19.7166666666667} 1.183 & \cellcolor{ForestGreen!16.6666666666667} 1 & \cellcolor{ForestGreen!33.3333333333333} 2 & \cellcolor{ForestGreen!35.75} 2.145 & \cellcolor{ForestGreen!33.3333333333333} 2 & \cellcolor{ForestGreen!32.4833333333333} 1.949 & \cellcolor{ForestGreen!50} 3 & \cellcolor{ForestGreen!48.8833333333333} 2.933 & \cellcolor{ForestGreen!65.4} 3.924 \\
    2-0   & \cellcolor{ForestGreen!33.3333333333333} 2 & \cellcolor{ForestGreen!32.4833333333333} 1.949 & \cellcolor{ForestGreen!33.3333333333333} 2 & \cellcolor{ForestGreen!35.75} 2.145 & \cellcolor{ForestGreen!16.6666666666667} 1 & \cellcolor{ForestGreen!16.6666666666667} 1 & \cellcolor{ForestGreen!19.7166666666667} 1.183 & \cellcolor{ForestGreen!0} 0 & \cellcolor{ForestGreen!7.45} 0.447 & \cellcolor{ForestGreen!16.6666666666667} 1 & \cellcolor{ForestGreen!48.8833333333333} 2.933 & \cellcolor{ForestGreen!48.8833333333333} 2.933 & \cellcolor{ForestGreen!50} 3 & \cellcolor{ForestGreen!64.9833333333333} 3.899 & \cellcolor{ForestGreen!65.4} 3.924 & \cellcolor{ForestGreen!81.3166666666667} 4.879 \\
    3-1   & \cellcolor{ForestGreen!35.75} 2.145 & \cellcolor{ForestGreen!33.3333333333333} 2 & \cellcolor{ForestGreen!32.4833333333333} 1.949 & \cellcolor{ForestGreen!33.3333333333333} 2 & \cellcolor{ForestGreen!19.7166666666667} 1.183 & \cellcolor{ForestGreen!16.6666666666667} 1 & \cellcolor{ForestGreen!16.6666666666667} 1 & \cellcolor{ForestGreen!7.45} 0.447 & \cellcolor{ForestGreen!0} 0 & \cellcolor{ForestGreen!16.6666666666667} 1 & \cellcolor{ForestGreen!50} 3 & \cellcolor{ForestGreen!48.8833333333333} 2.933 & \cellcolor{ForestGreen!48.8833333333333} 2.933 & \cellcolor{ForestGreen!65.4} 3.924 & \cellcolor{ForestGreen!64.9833333333333} 3.899 & \cellcolor{ForestGreen!81.3166666666667} 4.879 \\
    3-0   & \cellcolor{ForestGreen!50} 3 & \cellcolor{ForestGreen!48.8833333333333} 2.933 & \cellcolor{ForestGreen!48.8833333333333} 2.933 & \cellcolor{ForestGreen!50} 3 & \cellcolor{ForestGreen!33.3333333333333} 2 & \cellcolor{ForestGreen!32.4833333333333} 1.949 & \cellcolor{ForestGreen!33.3333333333333} 2 & \cellcolor{ForestGreen!16.6666666666667} 1 & \cellcolor{ForestGreen!16.6666666666667} 1 & \cellcolor{ForestGreen!0} 0 & \cellcolor{ForestGreen!65.4} 3.924 & \cellcolor{ForestGreen!64.9833333333333} 3.899 & \cellcolor{ForestGreen!65.4} 3.924 & \cellcolor{ForestGreen!81.3166666666667} 4.879 & \cellcolor{ForestGreen!81.3166666666667} 4.879 & \cellcolor{ForestGreen!97.4666666666667} 5.848 \\ \hline
    0-1   & \cellcolor{ForestGreen!16.6666666666667} 1 & \cellcolor{ForestGreen!16.6666666666667} 1 & \cellcolor{ForestGreen!19.7166666666667} 1.183 & \cellcolor{ForestGreen!24.7166666666667} 1.483 & \cellcolor{ForestGreen!32.4833333333333} 1.949 & \cellcolor{ForestGreen!33.3333333333333} 2 & \cellcolor{ForestGreen!35.75} 2.145 & \cellcolor{ForestGreen!48.8833333333333} 2.933 & \cellcolor{ForestGreen!50} 3 & \cellcolor{ForestGreen!65.4} 3.924 & \cellcolor{ForestGreen!0} 0 & \cellcolor{ForestGreen!7.45} 0.447 & \cellcolor{ForestGreen!14.9} 0.894 & \cellcolor{ForestGreen!16.6666666666667} 1 & \cellcolor{ForestGreen!19.7166666666667} 1.183 & \cellcolor{ForestGreen!33.3333333333333} 2 \\
    1-2   & \cellcolor{ForestGreen!19.7166666666667} 1.183 & \cellcolor{ForestGreen!16.6666666666667} 1 & \cellcolor{ForestGreen!16.6666666666667} 1 & \cellcolor{ForestGreen!19.7166666666667} 1.183 & \cellcolor{ForestGreen!33.3333333333333} 2 & \cellcolor{ForestGreen!32.4833333333333} 1.949 & \cellcolor{ForestGreen!33.3333333333333} 2 & \cellcolor{ForestGreen!48.8833333333333} 2.933 & \cellcolor{ForestGreen!48.8833333333333} 2.933 & \cellcolor{ForestGreen!64.9833333333333} 3.899 & \cellcolor{ForestGreen!7.45} 0.447 & \cellcolor{ForestGreen!0} 0 & \cellcolor{ForestGreen!7.45} 0.447 & \cellcolor{ForestGreen!16.6666666666667} 1 & \cellcolor{ForestGreen!16.6666666666667} 1 & \cellcolor{ForestGreen!32.4833333333333} 1.949 \\
    2-3   & \cellcolor{ForestGreen!24.7166666666667} 1.483 & \cellcolor{ForestGreen!19.7166666666667} 1.183 & \cellcolor{ForestGreen!16.6666666666667} 1 & \cellcolor{ForestGreen!16.6666666666667} 1 & \cellcolor{ForestGreen!35.75} 2.145 & \cellcolor{ForestGreen!33.3333333333333} 2 & \cellcolor{ForestGreen!32.4833333333333} 1.949 & \cellcolor{ForestGreen!50} 3 & \cellcolor{ForestGreen!48.8833333333333} 2.933 & \cellcolor{ForestGreen!65.4} 3.924 & \cellcolor{ForestGreen!14.9} 0.894 & \cellcolor{ForestGreen!7.45} 0.447 & \cellcolor{ForestGreen!0} 0 & \cellcolor{ForestGreen!19.7166666666667} 1.183 & \cellcolor{ForestGreen!16.6666666666667} 1 & \cellcolor{ForestGreen!33.3333333333333} 2 \\
    0-2   & \cellcolor{ForestGreen!33.3333333333333} 2 & \cellcolor{ForestGreen!32.4833333333333} 1.949 & \cellcolor{ForestGreen!33.3333333333333} 2 & \cellcolor{ForestGreen!35.75} 2.145 & \cellcolor{ForestGreen!48.8833333333333} 2.933 & \cellcolor{ForestGreen!48.8833333333333} 2.933 & \cellcolor{ForestGreen!50} 3 & \cellcolor{ForestGreen!64.9833333333333} 3.899 & \cellcolor{ForestGreen!65.4} 3.924 & \cellcolor{ForestGreen!81.3166666666667} 4.879 & \cellcolor{ForestGreen!16.6666666666667} 1 & \cellcolor{ForestGreen!16.6666666666667} 1 & \cellcolor{ForestGreen!19.7166666666667} 1.183 & \cellcolor{ForestGreen!0} 0 & \cellcolor{ForestGreen!7.45} 0.447 & \cellcolor{ForestGreen!16.6666666666667} 1 \\
    1-3   & \cellcolor{ForestGreen!35.75} 2.145 & \cellcolor{ForestGreen!33.3333333333333} 2 & \cellcolor{ForestGreen!32.4833333333333} 1.949 & \cellcolor{ForestGreen!33.3333333333333} 2 & \cellcolor{ForestGreen!50} 3 & \cellcolor{ForestGreen!48.8833333333333} 2.933 & \cellcolor{ForestGreen!48.8833333333333} 2.933 & \cellcolor{ForestGreen!65.4} 3.924 & \cellcolor{ForestGreen!64.9833333333333} 3.899 & \cellcolor{ForestGreen!81.3166666666667} 4.879 & \cellcolor{ForestGreen!19.7166666666667} 1.183 & \cellcolor{ForestGreen!16.6666666666667} 1 & \cellcolor{ForestGreen!16.6666666666667} 1 & \cellcolor{ForestGreen!7.45} 0.447 & \cellcolor{ForestGreen!0} 0 & \cellcolor{ForestGreen!16.6666666666667} 1 \\
    0-3   & \cellcolor{ForestGreen!50} 3 & \cellcolor{ForestGreen!48.8833333333333} 2.933 & \cellcolor{ForestGreen!48.8833333333333} 2.933 & \cellcolor{ForestGreen!50} 3 & \cellcolor{ForestGreen!65.4} 3.924 & \cellcolor{ForestGreen!64.9833333333333} 3.899 & \cellcolor{ForestGreen!65.4} 3.924 & \cellcolor{ForestGreen!81.3166666666667} 4.879 & \cellcolor{ForestGreen!81.3166666666667} 4.879 & \cellcolor{ForestGreen!97.4666666666667} 5.848 & \cellcolor{ForestGreen!33.3333333333333} 2 & \cellcolor{ForestGreen!32.4833333333333} 1.949 & \cellcolor{ForestGreen!33.3333333333333} 2 & \cellcolor{ForestGreen!16.6666666666667} 1 & \cellcolor{ForestGreen!16.6666666666667} 1 & \cellcolor{ForestGreen!0} 0 \\ \toprule
    \end{tabular}
    \begin{tablenotes} 
\item
Darker colour indicates a higher value. 
    \end{tablenotes}
\end{threeparttable}
}
\end{table}


The distances between the outcomes defined by this metric can be seen in Table~\ref{Table5} if $\pi = 9/10$ (they have a Pearson correlation over $0.99$ with the distances under $\pi = 4/5$ and $\pi = 19/20$). For example, the error derived in equation~\eqref{eq6} ($\sqrt{8.6} \approx 2.933$) can be found at the intersections of 2-0 and 1-2 (eighth row, twelfth column; twelfth row, eighth column) since formula~\eqref{eq5} is symmetric. Note that the same difference in goals scored results in the same distance, regardless of the number of goals scored: the distance of 1-0 and 2-1 equals the distance of 2-0 and 3-1, as well as the distance of 0-0 and 3-2 equals the distance of 1-0 and 3-3.
The measure is called \emph{distance of match scores} in the following.

\begin{table}[t!]
\centering
\caption{Average distances of match scores \eqref{eq5} for the statistical models}
\label{Table6}
\centerline{
\begin{threeparttable}
\rowcolors{3}{}{gray!20}
\begin{footnotesize}
    \begin{tabularx}{1.18\textwidth}{l CCC CCC CCC} \toprule \hiderowcolors
    \multirow{2}[0]{*}{Model} & \multicolumn{9}{c}{Season(s)} \\
          & \multicolumn{3}{c}{2003/04--2019/20} & \multicolumn{3}{c}{2020/21} & \multicolumn{3}{c}{2021/22} \\
    Parameter $\pi$ & \multicolumn{1}{c}{4/5} & \multicolumn{1}{c}{9/10} & \multicolumn{1}{c}{19/20} & \multicolumn{1}{c}{4/5} & \multicolumn{1}{c}{9/10} & \multicolumn{1}{c}{19/20} & \multicolumn{1}{c}{4/5} & \multicolumn{1}{c}{9/10} & \multicolumn{1}{c}{19/20} \\ \bottomrule \showrowcolors
    6p coeff & 1.986 (3) & 2.041 (3) & 1.923 (2) & 2.060 (3) & 2.199 (3) & 2.199 (3) & 2.095 (3) & 2.163 (3) & 2.038 (3) \\
    4p coeff & 1.978 (2) & 1.958 (2) & 1.922 (1) & 2.105 (5) & 2.068 (2) & 2.068 (2) & 2.118 (5) & 2.013 (2) & 2.043 (5) \\
    6p pot & 2.008 (5) & 2.054 (5) & 1.955 (5) & 2.053 (2) & 2.218 (5) & 2.218 (5) & 2.062 (2) & 2.175 (5) & 2.009 (2) \\
    4p pot & 1.954 (1) & 1.957 (1) & 1.954 (4) & 2.022 (1) & 2.066 (1) & 2.066 (1) & 2.061 (1) & 2.012 (1) & 2.008 (1) \\
    7p bivariate coeff & 1.986 (3) & 2.041 (3) & 1.923 (2) & 2.060 (3) & 2.199 (3) & 2.199 (3) & 2.095 (3) & 2.163 (3) & 2.038 (3) \\ \toprule
    Baseline & 2.069 (6) & 2.095 (6) & 2.082 (6) & 2.201 (6) & 2.247 (6) & 2.233 (6) & 2.178 (6) & 2.210 (6) & 2.132  (6) \\ \bottomrule
    \end{tabularx}
\end{footnotesize}
\begin{tablenotes} \footnotesize
\item
\emph{Baseline model}: The probability of any match outcome is determined by the relative frequency of this result in the training set (all seasons between 2003/04 and 2019/20).
\item
The ranks of the models are indicated in bracket.
\end{tablenotes}
\end{threeparttable}
}
\end{table}

Table~\ref{Table6} evaluates the six statistical models (including the baseline) according to the average distances of match scores over three sets of games. In contrast to the average hit probability, now a lower value is preferred. There is only a slight difference between the performance of variants based on UEFA club coefficients and pot allocation---and the latter does not seem to provide a worse estimation. Analogously, using six parameters instead of four does not improve the accuracy of the model. Since the schedule of group matches can be made dependent on the pots of the teams and UEFA club coefficients are not able to increase the predictive power, we have decided for the 4-parameter Poisson model based on pot allocation to simulate the group matches played in the Champions League. Note that this model has the best out-of-sample performance for all the three values of parameter $\pi$ considered.

However, the distance of match scores does not take the outcome (home win/draw/away win) into account. For example, the distances of 2-1 to 1-1 and to 3-1 are the same, but it can be argued that 3-1 is more similar to 2-1 since both results mean three points for the home team. Hence, formula~\eqref{eq5} can be modified as follows:
\begin{equation} \label{eq7}
\Theta \left( R_1, R_2 \right) = \Delta \left( R_1, R_2 \right) + \chi \left( R_1, R_2 \right),
\end{equation}
where
\[
\chi \left( R_1, R_2 \right) = 
\begin{cases}
0 & \text{if } \left( h_1 > a_1 \right) \land \left( h_2 > a_2 \right) \\
0 & \text{if } \left( h_1 = a_1 \right) \land \left( h_2 = a_2 \right) \\
0 & \text{if } \left( h_1 < a_1 \right) \land \left( h_2 < a_2 \right) \\
1 & \text{if } \left( h_1 > a_1 \right) \land \left( h_2 = a_2 \right) \\
1 & \text{if } \left( h_1 < a_1 \right) \land \left( h_2 = a_2 \right) \\
1 & \text{if } \left( h_1 = a_1 \right) \land \left( h_2 > a_2 \right) \\
1 & \text{if } \left( h_1 = a_1 \right) \land \left( h_2 < a_2 \right) \\
2 & \text{if } \left( h_1 > a_1 \right) \land \left( h_2 < a_2 \right) \\
2 & \text{if } \left( h_1 < a_1 \right) \land \left( h_2 > a_2 \right). \\
\end{cases}
\]
Thus, the term $\chi \left( R_1, R_2 \right)$ equals zero if the predicted and actual outcomes coincide, one if the prediction makes a small mistake (draw instead of home/away win and vice versa), and two if the prediction is the worst (home/away loss instead of home/away win).
For instance, if the final score is 2-1, then the error from equation~\eqref{eq7} is the same ($3$) for the forecasts 1-1 and 5-1.
This measure is called \emph{distance of match scores and outcomes}.

\begin{table}[t!]
\centering
\caption{Average distances of match scores and outcomes \eqref{eq7} for the statistical models}
\label{Table7}
\begin{threeparttable}
\rowcolors{3}{gray!20}{}
    \begin{tabularx}{0.8\textwidth}{lc CC} \toprule \hiderowcolors
    \multirow{2}[0]{*}{Model} & \multicolumn{3}{c}{Season(s)} \\
          & 2003/04--2019/20 & \multicolumn{1}{c}{2020/21} & 2021/22 \\ \bottomrule \showrowcolors
    6p coeff & 2.764 (4) & 2.941 (3) & 2.914 (5) \\
    4p coeff & 2.676 (1) & 2.821 (2) & 2.748 (2) \\
    6p pot & 2.788 (5) & 2.991 (5) & 2.885 (3) \\
    4p pot & 2.687 (2) & 2.820 (1) & 2.714 (1) \\
    7p bivariate coeff & 2.764 (4) & 2.941 (3) & 2.913 (5) \\ \toprule
    Baseline & 2.916 (6) & 3.131 (6) & 3.061 (6) \\ \bottomrule
    \end{tabularx}
\begin{tablenotes} \footnotesize
\item
\emph{Baseline model}: The probability of any match outcome is determined by the relative frequency of this result in the training set (all seasons between 2003/04 and 2019/20).
\item
The ranks of the models are indicated in bracket.
\item
The value of parameter $\pi$ is $9/10$.
\end{tablenotes}
\end{threeparttable}
\end{table}

Table~\ref{Table7} reports the average distances of match scores and outcomes according to the modified formula~\eqref{eq7} for $\pi = 9/10$. The main message remains unchanged: it makes no sense to use a model having more than four parameters, and our chosen variant 4p pot is competitive with 4p coeff, however, it is simpler and more convenient for scheduling purposes.

\input{Figure2_distribution_of_goals}

\begin{table}[t!]
\caption{The number of matches with a given outcome in the UEFA Champions League and the expected number of matches from the selected simulation model (4p pot)}
\label{Table8}

\begin{subtable}{\linewidth}
\centering
\caption{Seasons between 2003/04 and 2019/20}
\label{Table8a}
\rowcolors{1}{}{gray!20}
    \begin{tabularx}{\linewidth}{l CCCCC} \toprule
    Final score & 0     & 1     & 2     & 3     & 4 \\ \bottomrule
    0     & $\quad$ 115 $\quad$ ($103.0 \pm 9.8$) & $\quad$ 109 $\quad$ ($124.6 \pm 10.7$) & $\quad$ 83 $\quad$ ($81.2 \pm 8.7$) & $\quad$ 48 $\quad$ ($38.1 \pm 6.0$) & $\quad$ 20 $\quad$ ($14.0 \pm 3.7$) \\
    1     & $\quad$ 157 $\quad$ ($159.4 \pm 12.0$) & $\quad$ 175 $\quad$ ($175.8 \pm 12.5$) & $\quad$ 96 $\quad$ ($106.4 \pm 9.9$) & $\quad$ 38 $\quad$ ($46.8 \pm 6.7$) & $\quad$ 19 $\quad$ ($16.6 \pm 4.0$) \\
    2     & $\quad$ 138 $\quad$ ($133.8 \pm 11.0$) & $\quad$ 139 $\quad$ ($135.2 \pm 11.1$) & $\quad$ 76 $\quad$ ($74.7 \pm 8.4$) & $\quad$ 25 $\quad$ ($30.3 \pm 5.4$) & $\quad$ 5 $\quad$ ($9.9 \pm 3.1$) \\
    3     & $\quad$ 84 $\quad$ ($80.5 \pm 8.6$) & $\quad$ 72 $\quad$ ($75.6 \pm 8.5$) & $\quad$ 36 $\quad$ ($38.5 \pm 6.1$) & $\quad$ 14 $\quad$ ($14.3 \pm 3.8$) & $\quad$ 2 $\quad$ ($4.1 \pm 2.0$) \\
    4     & $\quad$ 44 $\quad$ ($38.9 \pm 6.1$) & $\quad$ 22 $\quad$ ($34.0 \pm 5.7$) & $\quad$ 18 $\quad$ ($16.0 \pm 4.0$) & $\quad$ 5 $\quad$ ($5.4 \pm 2.3$) & $\quad$ 2 $\quad$ ($1.6 \pm 1.3$) \\ \toprule
    \end{tabularx}
\end{subtable}

\vspace{0.25cm}
\begin{subtable}{\linewidth}
\centering
\caption{Season 2020/21}
\label{Table8b}
\rowcolors{1}{}{gray!20}
    \begin{tabularx}{\linewidth}{l CCCCC} \toprule
    Final score & 0     & 1     & 2     & 3     & 4 \\ \bottomrule
    0     & $\quad$ 5 $\quad$ ($6.1 \pm 2.4$) & $\quad$ 3 $\quad$ ($7.3 \pm 2.6$) & $\quad$ 8 $\quad$ ($4.8 \pm 2.1$) & $\quad$ 4 $\quad$ ($2.2 \pm 1.5$) & $\quad$ 4 $\quad$ ($0.8 \pm 0.9$) \\
    1     & $\quad$ 6 $\quad$ ($9.4 \pm 2.9$) & $\quad$ 9 $\quad$ ($10.3 \pm 3.0$) & $\quad$ 7 $\quad$ ($6.3 \pm 2.4$) & $\quad$ 3 $\quad$ ($2.8 \pm 1.6$) & $\quad$ 1 $\quad$ ($1.0 \pm 1.0$) \\
    2     & $\quad$ 7 $\quad$ ($7.9 \pm 2.7$) & $\quad$ 5 $\quad$ ($8.0 \pm 2.7$) & $\quad$ 6 $\quad$ ($4.4 \pm 2.0$) & $\quad$ 2 $\quad$ ($1.8 \pm 1.3$) & $\quad$ 0 $\quad$ ($0.6 \pm 0.8$) \\
    3     & $\quad$ 7 $\quad$ ($4.7 \pm 2.1$) & $\quad$ 5 $\quad$ ($4.4 \pm 2.1$) & $\quad$ 4 $\quad$ ($2.3 \pm 1.5$) & $\quad$ 0 $\quad$ ($0.8 \pm 0.9$) & $\quad$ 1 $\quad$ ($0.2 \pm 0.5$) \\
    4     & $\quad$ 2 $\quad$ ($2.3 \pm 1.5$) & $\quad$ 1 $\quad$ ($2.0 \pm 1.4$) & $\quad$ 0 $\quad$ ($0.9 \pm 1.0$) & $\quad$ 0 $\quad$ ($0.3 \pm 0.6$) & $\quad$ 0 $\quad$ ($0.1 \pm 0.3$) \\ \toprule
    \end{tabularx}
\end{subtable}

\vspace{0.25cm}
\begin{subtable}{\linewidth}
\centering
\caption{Season 2021/22}
\label{Table8c}
\begin{threeparttable}
\rowcolors{1}{}{gray!20}
    \begin{tabularx}{\linewidth}{l CCCCC} \toprule
    Final score & 0     & 1     & 2     & 3     & 4 \\ \bottomrule
    0     & $\quad$ 6 $\quad$ ($6.1 \pm 2.4$) & $\quad$ 5 $\quad$ ($7.3 \pm 2.6$) & $\quad$ 1 $\quad$ ($4.8 \pm 2.1$) & $\quad$ 3 $\quad$ ($2.2 \pm 1.5$) & $\quad$ 1 $\quad$ ($0.8 \pm 0.9$) \\
    1     & $\quad$ 9 $\quad$ ($9.4 \pm 2.9$) & $\quad$ 7 $\quad$ ($10.3 \pm 3.0$) & $\quad$ 8 $\quad$ ($6.3 \pm 2.4$) & $\quad$ 4 $\quad$ ($2.8 \pm 1.6$) & $\quad$ 2 $\quad$ ($1.0 \pm 1.0$) \\
    2     & $\quad$ 10 $\quad$ ($7.9 \pm 2.7$) & $\quad$ 7 $\quad$ ($8.0 \pm 2.7$) & $\quad$ 3 $\quad$ ($4.4 \pm 2.0$) & $\quad$ 2 $\quad$ ($1.8 \pm 1.3$) & $\quad$ 0 $\quad$ ($0.6 \pm 0.8$) \\
    3     & $\quad$ 2 $\quad$ ($4.7 \pm 2.1$) & $\quad$ 3 $\quad$ ($4.4 \pm 2.1$) & $\quad$ 3 $\quad$ ($2.3 \pm 1.5$) & $\quad$ 2 $\quad$ ($0.8 \pm 0.9$) & $\quad$ 0 $\quad$ ($0.2 \pm 0.5$) \\
    4     & $\quad$ 5 $\quad$ ($2.3 \pm 1.5$) & $\quad$ 2 $\quad$ ($2.0 \pm 1.4$) & $\quad$ 2 $\quad$ ($0.9 \pm 1.0$) & $\quad$ 0 $\quad$ ($0.3 \pm 0.6$) & $\quad$ 0 $\quad$ ($0.1 \pm 0.3$) \\ \toprule
    \end{tabularx}
\begin{tablenotes} \footnotesize
\item
Goals scored by the home team are in the rows, goals scored by the away team are in the columns.
\item
Games where one team scored at least five goals are not presented.
\item
The numbers in parenthesis indicate the average number of occurrences based on simulations $\pm$ standard deviations.
\end{tablenotes}
\end{threeparttable}
\end{subtable}
\end{table}

Finally, the chosen specification is demonstrated to describe well the unknown score-generating process.
First, Figure~\ref{Fig2} shows the real goal distributions and the one implied by our Poisson model that gives the same forecast for each season as the teams are identified by the pot from which they are drawn.
Second, the final scores of the games are analysed: Table~\ref{Table8} presents the number of matches with the given outcome in the corresponding season(s) and the number of occurrences for these events according to the chosen simulation model. Again, the forecast is the same for any season since the groups cannot be distinguished by the strengths of the clubs.

To conclude, the 4-parameter Poisson model based on pot allocation provides a good approximation to the empirical data. This is essential for further analysis: since each group contains one team from each pot, the groups are identical with respect to our simulation model. Consequently, the performance of any schedule is the same for the game classification scheme. Otherwise, the predicted probability of a (weakly/strongly) stakeless match might depend on other characteristics of the clubs playing the group, for instance, their exact UEFA club coefficients, which are not known before the group draw (although the coefficients of all the 32 teams are naturally known, it remains uncertain what teams will play in a particular group).

\subsection{Feasible schedules} \label{Sec42}

Section~\ref{Sec41} provides a tool to simulate the group stage of the UEFA Champions League. However, we are interested in how stakeless games can be avoided, for which purpose the schedule of the matches can be chosen by the organiser subject to some constraints.

The regulation of the UEFA Champions League provides surprisingly little information on how the group matches are scheduled \citep[Article~16.02]{UEFA2021c}: ``\emph{A club does not play more than two home or two away matches in a row and each club plays one home match and one away match on the first and last two matchdays.}''
In the Champions League seasons from 2003/04 to 2020/21, Matchday 4/5/6 was the mirror image of Matchday 3/1/2, respectively. Consequently, the same two teams played at home in the first and last rounds since Matchday~6 mirrors Matchday~2, where the two teams that play at home in Matchday~1 have an away game. The arrangement has been changed from the 2021/22 season such that Matchday 4/5/6 is the mirror image of Matchday 3/2/1. However, this is not described in the regulation, we have explored it only through ``reverse engineering''.

\begin{table}[t!]
\centering
\caption{Valid schedules for the UEFA Champions League groups and \\ the number of groups using them between 2018/19 and 2022/23}
\label{Table9}
\begin{threeparttable}
\rowcolors{3}{gray!20}{}
    \begin{tabularx}{0.8\textwidth}{l CCCC C} \toprule \hiderowcolors
    \multicolumn{1}{c}{\multirow{2}[0]{*}{Schedule}} & \multicolumn{2}{c}{Matchday 5} & \multicolumn{2}{c}{Matchday 6} & \multicolumn{1}{c}{\multirow{2}[0]{*}{Groups}} \\
    	  & Home  & Away  & Home  & Away & \\ \bottomrule \showrowcolors
    1231  & 1     & 2     & 3     & 1     & 10 \\
          & 4     & 3     & 2     & 4     &  \\ \hline
    2113  & 2     & 1     & 1     & 3     & 2 \\
          & 3     & 4     & 4     & 2     &  \\ \hline
    1241  & 1     & 2     & 4     & 1     & 3 \\
          & 3     & 4     & 2     & 3     &  \\ \hline
    2114  & 2     & 1     & 1     & 4     & 5 \\
          & 4     & 3     & 3     & 2     &  \\ \hline
    1321  & 1     & 3     & 2     & 1     & 3 \\
          & 4     & 2     & 3     & 4     &  \\ \hline
    3112  & 3     & 1     & 1     & 2     & 2 \\
          & 2     & 4     & 4     & 3     &  \\ \hline
    1341  & 1     & 3     & 4     & 1     & 5 \\
          & 2     & 4     & 3     & 2     &  \\ \hline
    3114  & 3     & 1     & 1     & 4     & 2 \\
          & 4     & 2     & 2     & 3     &  \\ \hline
    1421  & 1     & 4     & 2     & 1     & 0 \\
          & 3     & 2     & 4     & 3     &  \\ \hline
    4112  & 4     & 1     & 1     & 2     & 4 \\
          & 2     & 3     & 3     & 4     &  \\ \hline
    1431  & 1     & 4     & 3     & 1     & 3 \\
          & 2     & 3     & 4     & 2     &  \\ \hline
    4113  & 4     & 1     & 1     & 3     & 1 \\
          & 3     & 2     & 2     & 4     &  \\ \bottomrule
    \end{tabularx}
    
    \begin{tablenotes} \footnotesize
\item
The numbers in the four middle columns indicate the pots from which the teams are drawn. 
\item
The numbers in the last column show the number of UEFA Champions League groups where the given schedule has been followed in the five seasons between 2018/19 and 2022/23.
    \end{tablenotes}
\end{threeparttable}
\end{table}

Therefore, there are 12 valid schedules for the last two matchdays as an arbitrarily chosen team has three possible opponents for Matchday~5, two possible opponents for Matchday~6 (it cannot play against the same team as in Matchday~5), and can play at home either in Matchday~5 or Matchday~6. The derived options are listed in Table~\ref{Table9}, together with their prevalence in the five seasons between 2018/19 and 2022/23. It seems that all of these options can be accepted by the UEFA.
Since each group consists of one team from each of the four pots, the clubs are identified by their pot in the following, that is, team $i$ represents the team drawn from Pot $i$.

According to Section~\ref{Sec41}, the best simulation model is the 4-parameter Poisson based on pot allocation (4p pot), which will be used to derive the numerical results. In particular, all group matches are simulated 1 million times, but the results of the matches played in the last matchday(s) are disregarded to calculate the probability of a stakeless game under the given group schedule. For sample size $N$, the error of a simulated probability $P$ is $\sqrt{P(1-P) / N}$. Since even the smallest $P$ exceeds 2.5\% and 1 million simulation runs are implemented, the error always remains below 0.016\%. Therefore, confidence intervals will not be provided because the averages differ reliably between the possible schedules.

\subsection{Results} \label{Sec43}

\begin{table}[t!]
\centering
\caption{The probability of stakeless games}
\label{Table10}
\begin{threeparttable}
\rowcolors{3}{gray!20}{}
    \begin{tabularx}{\textwidth}{l CCC} \toprule \hiderowcolors
    \multirow{2}[0]{*}{Schedule} & \multicolumn{3}{c}{The probability of stakeless games in percentage (\%)} \\
          & Weakly in Matchday 5 & Weakly in Matchday 6 & Strongly in Matchday 6 \\ \bottomrule \showrowcolors
    1231  & 2.62 (3)\textcolor{gray!20}{0} & 35.37 (8)\textcolor{gray!20}{0} & \textcolor{gray!20}{0}8.02 (5)\textcolor{gray!20}{0} \\
    2113  & 2.60 (1)\textcolor{white}{0} & 35.40 (9)\textcolor{white}{0} & \textcolor{white}{0}7.94 (4)\textcolor{white}{0} \\
    1241  & 2.84 (7)\textcolor{gray!20}{0} & 35.73 (10) 					 & \textcolor{gray!20}{0}8.40 (6)\textcolor{gray!20}{0} \\
    2114  & 2.82 (5)\textcolor{white}{0} & 34.79 (7)\textcolor{white}{0} & \textcolor{white}{0}7.70 (3)\textcolor{white}{0} \\
    1321  & 2.60 (1)\textcolor{gray!20}{0} & 28.41 (2)\textcolor{gray!20}{0} & 10.16 (11) \\
    3112  & 2.62 (3)\textcolor{white}{0} & 28.49 (3)\textcolor{white}{0} & 10.13 (10) \\
    1341  & 3.97 (11) 					 & 36.78 (11)					 & \textcolor{gray!20}{0}8.85 (8)\textcolor{gray!20}{0} \\
    3114  & 3.96 (9)\textcolor{white}{0} & 33.65 (6)\textcolor{white}{0} & \textcolor{white}{0}7.31 (2)\textcolor{white}{0} \\
    1421  & 2.82 (5)\textcolor{gray!20}{0} & 29.95 (4)\textcolor{gray!20}{0} & 10.63 (12) \\
    4112  & 2.84 (7)\textcolor{white}{0} & 26.92 (1)\textcolor{white}{0} & \textcolor{white}{0}9.68 (9)\textcolor{white}{0} \\
    1431  & 3.96 (9)\textcolor{gray!20}{0} & 37.36 (12)					 & \textcolor{gray!20}{0}8.82 (7)\textcolor{gray!20}{0} \\
    4113  & 3.97 (11) 					 & 33.34 (5)\textcolor{white}{0} & \textcolor{white}{0}7.22 (1)\textcolor{white}{0} \\ \bottomrule
    \end{tabularx}
\begin{tablenotes} \footnotesize
\item
The ranks of the schedules are indicated in bracket.
\end{tablenotes}
\end{threeparttable}
\end{table}

Table~\ref{Table10} shows the likelihood of a stakeless game in a given round as a function of the schedule.
The probability of a weakly stakeless game in Matchday 5 varies approximately between 2.6\% and 4\%. It is the lowest for schedules 2113 and 1321, while schedules 1341 and 4113 are poor choices to avoid these matches. There are only six different values for the 12 schedules because, if the final rank of a team is already fixed after four rounds, then its match played in Matchday~5 will be weakly stakeless, independently of its opponent.

This remark does not refer to the probability of a weakly stakeless game played in the last round, which is, for example, above 35\% for schedule 1241 but below 27\% for schedule 4112 where the order of the last two matchdays is exchanged compared to schedule 1241. The frequency of weakly stakeless games is almost by an order of magnitude higher in Matchday~6 than in Matchday~5.
 
Finally, the probability of a strongly stakeless game is between 7.2\% and 10.6\%. It is the lowest for schedule 4113, which is the worst with respect to the probability of weakly stakeless games in Matchday~5.

To summarise, the schedule has a non-negligible role concerning the competitiveness of the matches. On the basis of Table~\ref{Table10}, the probability of a weakly stakeless game can be reduced by 35\% ($1-2.60/3.97$) in Matchday~5 and by 28\% ($1-26.92/37.36$) in Matchday 6, while the probability of a strongly stakeless game can decrease by 32\% ($1-7.22/10.63$) if the most favourable alternative is chosen instead of the worst.

Among the 12 schedules, the following four are dominated:
\begin{itemize}
\item
Schedule 1341 is worse than its ``inverse'' 4113, derived by exchanging the last two matchdays;
\item
Schedule 1431 is worse than its ``inverse'' 3114, derived by exchanging the last two matchdays;
\item
Schedule 1241 is worse than ``its pair'' 2114, derived by exchanging the home-away pattern in the last two matchdays;
\item
Schedule 1421 is worse than schedule 1321, where the top team plays against a stronger team in Matchday~5.
\end{itemize}
These schedules should not be used if the aim is to avoid stakeless games.

Since there are three different objectives, that is, to minimise the probability of weakly stakeless games played in Matchday 5 and 6, as well as the probability of strongly stakeless games, a weighting scheme can be chosen for these events to get an aggregate cost estimate for each schedule, which is a standard approach in multi-objective optimisation. The cost of a weakly stakeless game played in Matchday 6 can be 1 without losing generality, which calls for determining the relative cost of a weakly stakeless game played in Matchday 5 and a strongly stakeless game. The former is unlikely to be lower than 1. Regarding the latter, there are two contradictory arguments:
\begin{itemize}
\item
If a team has something to play for but its opponent has no such incentives (weakly stakeless game), then the match is exposed to the risk of manipulation or---depending on the outcome---to the impression that the match has been sold. This problem does not emerge if no team can improve its ranking (strongly stakeless game).
\item
Weakly stakeless games can generate attendance because at least one team should exert effort. On the other hand, strongly stakeless games might be completely boring and prone to betting-related manipulation (match fixing) as the targets of both contesting clubs are obtained or lost anyway, regardless of the match outcome they agree upon \citep{VanwerschWillemConstandtHardyns2022}.
\end{itemize}
Hence, the relative cost of weakly and strongly stakeless games greatly rests on the preferences of the decision-maker, and may differ even by an order of magnitude.

\begin{figure}[ht!]
\centering

\begin{subfigure}{\textwidth}
\captionsetup{justification=centering}
\caption{A strongly stakeless game is less costly than a weakly stakeless game}
\label{Fig3a}

\vspace{0.25cm}
\begin{tikzpicture}
\begin{axis}[
xlabel = The cost ratio of a strongly and a weakly stakeless game on Matchday 6,
x label style = {font=\small},
ylabel = The weighted cost of stakeless games,
y label style = {font=\small},
width = 0.99\textwidth,
height = 0.6\textwidth,
ymajorgrids = true,
xmin = 0.1,
xmax = 1,
max space between ticks=50,
] 

\addplot [red, thick, dashdotdotted, mark=pentagon*, mark options=solid] coordinates {
(0.1,0.38790695)
(0.2,0.3959299)
(0.3,0.40395285)
(0.4,0.4119758)
(0.5,0.41999875)
(0.6,0.4280217)
(0.7,0.43604465)
(0.8,0.4440676)
(0.9,0.45209055)
(1,0.4601135)
};

\addplot [gray, thick, mark=|, mark options=solid] coordinates {
(0.1,0.38793755)
(0.2,0.3958811)
(0.3,0.40382465)
(0.4,0.4117682)
(0.5,0.41971175)
(0.6,0.4276553)
(0.7,0.43559885)
(0.8,0.4435424)
(0.9,0.45148595)
(1,0.4594295)
};

\addplot [blue, thick, dashed, mark=asterisk, mark options={solid,semithick}] coordinates {
(0.1,0.3837505)
(0.2,0.391454)
(0.3,0.3991575)
(0.4,0.406861)
(0.5,0.4145645)
(0.6,0.422268)
(0.7,0.4299715)
(0.8,0.437675)
(0.9,0.4453785)
(1,0.453082)
};

\addplot [orange, thick, dashdotted, mark=square*, mark options={solid,thin}] coordinates {
(0.1,0.3202556)
(0.2,0.3304202)
(0.3,0.3405848)
(0.4,0.3507494)
(0.5,0.360914)
(0.6,0.3710786)
(0.7,0.3812432)
(0.8,0.3914078)
(0.9,0.4015724)
(1,0.411737)
};

\addplot [ForestGreen, thick, dashed, mark=triangle*, mark options=solid] coordinates {
(0.1,0.32122575)
(0.2,0.3313555)
(0.3,0.34148525)
(0.4,0.351615)
(0.5,0.36174475)
(0.6,0.3718745)
(0.7,0.38200425)
(0.8,0.392134)
(0.9,0.40226375)
(1,0.4123935)
};

\addplot [brown, thick, mark=diamond*, mark options=solid] coordinates {
(0.1,0.38341235)
(0.2,0.3907257)
(0.3,0.39803905)
(0.4,0.4053524)
(0.5,0.41266575)
(0.6,0.4199791)
(0.7,0.42729245)
(0.8,0.4346058)
(0.9,0.44191915)
(1,0.4492325)
};

\addplot [purple, thick, mark=oplus*, mark options=solid] coordinates {
(0.1,0.30723315)
(0.2,0.3169118)
(0.3,0.32659045)
(0.4,0.3362691)
(0.5,0.34594775)
(0.6,0.3556264)
(0.7,0.36530505)
(0.8,0.3749837)
(0.9,0.38466235)
(1,0.394341)
};

\addplot [black, thick, dotted, mark=star, mark size=3pt, mark options=solid] coordinates {
(0.1,0.38028805)
(0.2,0.3875086)
(0.3,0.39472915)
(0.4,0.4019497)
(0.5,0.40917025)
(0.6,0.4163908)
(0.7,0.42361135)
(0.8,0.4308319)
(0.9,0.43805245)
(1,0.445273)
};
\end{axis}
\end{tikzpicture}
\end{subfigure}

\vspace{0.5cm}
\begin{subfigure}{\textwidth}
\captionsetup{justification=centering}
\caption{A strongly stakeless game is more costly than a weakly stakeless game}
\label{Fig3b}

\vspace{0.25cm}
\begin{tikzpicture}
\begin{axis}[
xlabel = The cost ratio of a strongly and a weakly stakeless game in Matchday 6,
x label style = {font=\small},
ylabel = The weighted cost of stakeless games,
y label style = {font=\small},
width = 0.99\textwidth,
height = 0.6\textwidth,
ymajorgrids = true,
xmin = 1,
xmax = 10,
max space between ticks=50,
legend style = {font=\small,at={(-0.05,-0.15)},anchor=north west,legend columns=4},
legend entries = {Schedule 1231$\qquad$,Schedule 2113$\qquad$,Schedule 2114$\qquad$,Schedule 1321,Schedule 3112$\qquad$,Schedule 3114$\qquad$,Schedule 4112$\qquad$,Schedule 4113}
] 

\addplot [red, thick, dashdotdotted, mark=pentagon*, mark options=solid] coordinates {
(1,0.4601135)
(1.5,0.50022825)
(2,0.540343)
(2.5,0.58045775)
(3,0.6205725)
(3.5,0.66068725)
(4,0.700802)
(4.5,0.74091675)
(5,0.7810315)
(5.5,0.82114625)
(6,0.861261)
(6.5,0.90137575)
(7,0.9414905)
(7.5,0.98160525)
(8,1.02172)
(8.5,1.06183475)
(9,1.1019495)
(9.5,1.14206425)
(10,1.182179)
};

\addplot [gray, thick, mark=|, mark options=solid] coordinates {
(1,0.4594295)
(1.5,0.49914725)
(2,0.538865)
(2.5,0.57858275)
(3,0.6183005)
(3.5,0.65801825)
(4,0.697736)
(4.5,0.73745375)
(5,0.7771715)
(5.5,0.81688925)
(6,0.856607)
(6.5,0.89632475)
(7,0.9360425)
(7.5,0.97576025)
(8,1.015478)
(8.5,1.05519575)
(9,1.0949135)
(9.5,1.13463125)
(10,1.174349)
};

\addplot [blue, thick, dashed, mark=asterisk, mark options={solid,semithick}] coordinates {
(1,0.453082)
(1.5,0.4915995)
(2,0.530117)
(2.5,0.5686345)
(3,0.607152)
(3.5,0.6456695)
(4,0.684187)
(4.5,0.7227045)
(5,0.761222)
(5.5,0.7997395)
(6,0.838257)
(6.5,0.8767745)
(7,0.915292)
(7.5,0.9538095)
(8,0.992327)
(8.5,1.0308445)
(9,1.069362)
(9.5,1.1078795)
(10,1.146397)
};

\addplot [orange, thick, dashdotted, mark=square*, mark options={solid,thin}] coordinates {
(1,0.411737)
(1.5,0.46256)
(2,0.513383)
(2.5,0.564206)
(3,0.615029)
(3.5,0.665852)
(4,0.716675)
(4.5,0.767498)
(5,0.818321)
(5.5,0.869144)
(6,0.919967)
(6.5,0.97079)
(7,1.021613)
(7.5,1.072436)
(8,1.123259)
(8.5,1.174082)
(9,1.224905)
(9.5,1.275728)
(10,1.326551)
};

\addplot [ForestGreen, thick, dashed, mark=triangle*, mark options=solid] coordinates {
(1,0.4123935)
(1.5,0.46304225)
(2,0.513691)
(2.5,0.56433975)
(3,0.6149885)
(3.5,0.66563725)
(4,0.716286)
(4.5,0.76693475)
(5,0.8175835)
(5.5,0.86823225)
(6,0.918881)
(6.5,0.96952975)
(7,1.0201785)
(7.5,1.07082725)
(8,1.121476)
(8.5,1.17212475)
(9,1.2227735)
(9.5,1.27342225)
(10,1.324071)
};

\addplot [brown, thick, mark=diamond*, mark options=solid] coordinates {
(1,0.4492325)
(1.5,0.48579925)
(2,0.522366)
(2.5,0.55893275)
(3,0.5954995)
(3.5,0.63206625)
(4,0.668633)
(4.5,0.70519975)
(5,0.7417665)
(5.5,0.77833325)
(6,0.8149)
(6.5,0.85146675)
(7,0.8880335)
(7.5,0.92460025)
(8,0.961167)
(8.5,0.99773375)
(9,1.0343005)
(9.5,1.07086725)
(10,1.107434)
};

\addplot [purple, thick, mark=oplus*, mark options=solid] coordinates {
(1,0.394341)
(1.5,0.44273425)
(2,0.4911275)
(2.5,0.53952075)
(3,0.587914)
(3.5,0.63630725)
(4,0.6847005)
(4.5,0.73309375)
(5,0.781487)
(5.5,0.82988025)
(6,0.8782735)
(6.5,0.92666675)
(7,0.97506)
(7.5,1.02345325)
(8,1.0718465)
(8.5,1.12023975)
(9,1.168633)
(9.5,1.21702625)
(10,1.2654195)
};

\addplot [black, thick, dotted, mark=star, mark size=3pt, mark options=solid] coordinates {
(1,0.445273)
(1.5,0.48137575)
(2,0.5174785)
(2.5,0.55358125)
(3,0.589684)
(3.5,0.62578675)
(4,0.6618895)
(4.5,0.69799225)
(5,0.734095)
(5.5,0.77019775)
(6,0.8063005)
(6.5,0.84240325)
(7,0.878506)
(7.5,0.91460875)
(8,0.9507115)
(8.5,0.98681425)
(9,1.022917)
(9.5,1.05901975)
(10,1.0951225)
};
\end{axis}
\end{tikzpicture}
\end{subfigure}

\captionsetup{justification=centering}
\caption{The prices of the Champions League group stage schedules if a weakly \\ stakeless game has the same cost in Matchday 5 as in Matchday 6}
\label{Fig3}

\end{figure}


Figure~\ref{Fig3} calculates the price of schedules as a function of the cost ratio between a strongly and weakly stakeless game played in the last round if the weakly stakeless games have a uniform cost. Schedule 4112 is the best alternative if the cost of a strongly stakeless game is at most 3, however, schedule 4113 should be chosen if this ratio is higher.



The key findings can be summed up as follows:
\begin{itemize}
\item
Schedules 1321 and 3112, as well as schedules 1231 and 2113 are almost indistinguishable from the perspective of our objective as they vary only the home-away pattern in the last two matchdays;
\item
Schedule 4112 is better than schedules 1321 and 3112 except if weakly stakeless games in Matchday~5 have an unreasonably high relative cost;
\item
The other five non-dominated schedules (1231, 2113, 2114, 3114, 4113) perform similarly to each other; they imply a worse ratio of weakly stakeless games, but should be used to avoid strongly stakeless games;
\item
Schedules where the team drawn from Pot 1 plays at home in the last round are almost always better than schedules where this team plays its last match away;
\item
Schedules 4112 and 4113 seem to be the best alternatives with respect to competitiveness.
\end{itemize}

It is also important to provide an intuitive explanation for these results.
First, note that home advantage can be used to partially compensate for the intrinsic strength differences between the teams. Therefore, in Matchday~6, the strongest team should play at home and the weakest team away in order to guarantee two and three home matches in the first five rounds for them, respectively. Then the set of potential schedules is limited to 4112, 4113, 3114, and 2114. Among them, schedule 4112 minimises the difference between the strength of schedules both for the two strongest and the two weakest teams. This  is advantageous with respect to weakly stakeless games: according to Table~\ref{Table10}, the probability of a weakly stakeless game in Matchday~6 is the lowest if the two strongest teams play against each other in Mathchday~6.

On the other hand, schedule 4112 is worse than the other three with respect to strongly stakeless games as, if one of the two favourites has unexpected loss(es), or one of the two underdogs has unexpected win(s), the current position of their opponent can easily be secured. The problem with schedule 2114 is an away match for the second-ranked team and a home match for the third-ranked team in Matchday~6, meaning that home advantage is not fully exploited to increase competitiveness. The difference between the remaining schedules 4113 and 3114 is minimal for our metrics. Nonetheless, schedule 4113 is the better option since schedule 3114 contains the less uncertain game of the whole group (the strongest team plays at home against the weakest) in the last round. This game seems to be competitive only by our deterministic model that does not differentiate between the uncertainty of the matches still to be played.

To summarise, home advantage can be used to increase entertainment if the strongest (weakest) team plays at home (away) in the last round. Other considerations limit the set of competitive options to schedules 4112 and 4113; the former (latter) being optimal with respect to weakly (strongly) stakeless games.

\subsection{Limitations} \label{Sec44}

Naturally, our study has certain limitations. The simulation model may be refined and sensitivity analysis can be carried out with various assumptions on the outcomes of the games. The Elo rating of the teams may be a better predictor of the Champions League results, although we have followed the extant literature \citep{CoronaForrestTenaWiper2019, DagaevRudyak2019} by using UEFA club coefficients for this purpose.
In addition, other aspects of scheduling fairness such as balancing the kick-off times of the matches \citep{Krumer2020a} or the home games played on non-frequent days between the teams \citep{GollerKrumer2020} have been neglected. Our simulation technique is independent of the order of the matches but scheduling may affect performance: \citet{Krumer2021} finds that a shorter gap between the two matches favours the underdog team, especially if it did not lose in the first match. Last but not least, a team may suffer from certainly playing at home in the last round, which implies playing fewer home games in the first five rounds.

The cost of a weakly stakeless match may depend on whether the team whose position in the group ranking is already known plays at home or away. For example, in the basketball EuroLeague, the games played at the stadium of a team that ensured home advantage in the playoffs attract significantly fewer fans \citep{DiMattiaKrumer2023}. 

The suggested game classification scheme does not deal with the sequence of matches played in the first four rounds, but a change in the scheduling of Matchdays 1--4 might affect in-tournament dynamics, including psychological and behavioural reasons. Having fewer stakeless games in the last two matchdays can imply that the previous matches are less relevant and interesting for the fans: according to \citet{DiMattiaKrumer2023}, games played in the first part of a double round-robin tournament are likely perceived as less decisive than games played in the second part.

Stakeless games are identified in a deterministic framework with three distinct categories. However, a team may exert lower effort still if its position is known with a high probability. Since the difference between the value of the first two places is probably smaller in the UEFA Champions League groups than the difference between the value of the second and the third positions, and winning a match is awarded by the revenue distribution system \citep{UEFA2021b}, the true incentive scheme is not binary as in our model.

The number of stakeless games can also be reduced through more radical changes in the tournament design. For example, teams that have performed best during the preliminary group stage can choose their opponents during the subsequent knockout stage in order to provide a strong incentive for exerting full effort even if the position of the team in the final ranking is already known \citep{Guyon2022a}. Analogously, additional draw constraints can contribute to avoiding unfair situations \citep{Csato2022a}, which might include stakeless games.

Finally, the group fixtures are not necessarily independent of each other and there can be other restrictions. For instance, two Italian teams, FC Internazionale Milano and AC Milan that have played in the 2021/22 UEFA Champions League share the same stadium, thus, both of them cannot play at home on the same matchday. Clubs from certain countries are better to play at home in Matchday 5 instead of Matchday 6 due to weather conditions. Similar constraints might prevent choosing the optimal schedule for all groups.

\section{Discussion} \label{Sec5}

The paper has proposed a novel classification method for games played in a round-robin tournament. Our selection criterion is connected to the incentives of the teams, it depends on whether the position of a team in the final ranking is already known, independently of the outcomes of matches still to be played. In particular, a game is called (1) competitive if neither opposing team is indifferent; (2) weakly stakeless if exactly one of the opposing teams is indifferent; or (3) strongly stakeless if both teams are indifferent. Avoiding stakeless matches should be an imperative aim of the organiser because a team might play with little enthusiasm if the outcome of the match cannot affect its final position, which probably reduces attendance and is unfair to the teams that have already played against this particular team as illustrated by Example~\ref{Examp1}.

The group stage of the UEFA Champions League, the most prestigious European club football competition, is currently organised as a double round-robin tournament with four teams. Therefore, a simulation model has been built to compare the 12 possible sequences for the group matches with respect to the probability of games where one or both clubs cannot achieve a higher rank. Some schedules are shown to be dominated by other schedules from this perspective. It is found that the strongest team should play at home in the last round against one of the middle teams, depending on the preferences of the tournament organiser.

A better prediction model may be found in the future which is determined by other variables, making the best schedule dependent on the composition of the groups. However, choosing a uniform schedule for all groups has the advantage of being \emph{transparent}, and can prevent long debates on whether the proposed schedule has been manipulated in favour of a particular team.

Even though the UEFA Champions League will see a fundamental reform from the 2024/25 season, double round-robin contests with four teams are common in other sports and tournaments. For example, they are used in the UEFA Nations League, in the African section of the FIFA World Cup qualification, and in the qualification for the European Men's Handball Championship, among others. Then, our game classification scheme proposed in Section~\ref{Sec3} can be directly applied, while the simulation methodology described in Section~\ref{Sec4} can be updated with the relevant historical data and statistical model in order to determine the best schedule concerning the competitiveness of the games.

We have argued that a tournament schedule is \emph{fairer} if the probability of games where at least one team has few incentives to win is lower. These situations threaten with one team or both teams playing intentionally below their full potential, which might lead to the opponent scoring/not conceding a goal. Since the number of matches is fixed, reducing the probability of stakeless games maximises the expected number of competitive games that are more exciting to watch. The assessment and simulation method suggested here allows for the organiser to choose an optimal timetable without altering other characteristics of the tournament (number of teams, number of matches, qualification rules, points system, etc.). Consequently, the final ranking will better reflect the true strengths of the teams as remaining less affected by the unwanted incentives attributable to the tournament schedule. Having fewer stakeless games is also beneficial for the teams that are less likely to suffer from unfair results of matches played by other teams and for fans who can see more matches where both teams give their best.

Stakeless games may have powerful effects on stadium attendance and television audience. According to \citet{DiMattiaKrumer2023}, basketball clubs playing at home have smaller attendance demand after they have ensured home advantage in the playoffs. Unsurprisingly, a team still in contention to win the championship positively affects attendance demand \citep{PawlowskiNalbantis2015}. \citet{BuraimoForrestMcHaleTena2022} find that a match with the highest championship significance in the English Premier League is expected to attract a  96\% higher audience compared to a match without implications for end-of-season prizes but with the same characteristics (clubs, players, etc.). Consequently, exploring the relationship between attendance demand and stakeless games in the UEFA Champions League seems to be an interesting direction for future research.

To conclude, picking up an optimal sequence of games with respect to the proposed metric increases the utility of all stakeholders at a minimal price if the scheduling constraints are appropriately defined. Therefore, our study has hopefully managed to uncover an important aspect of tournament design and can inspire further research by scheduling experts to optimise various measures of competitiveness beyond the classical criteria of fairness.

\section*{Acknowledgements}
\addcontentsline{toc}{section}{Acknowledgements}
\noindent
This paper could not have been written without the \emph{father} of the first author (also called \emph{L\'aszl\'o Csat\'o}), who has helped to code the simulations in Python. \\
We are grateful to \emph{Dries Goossens}, \emph{Alex Krumer}, \emph{Frits C.~R.~Spieksma}, and \emph{Stephan Westphal} for useful advice. \\
Eight anonymous reviewers provided valuable comments and suggestions on earlier drafts. \\
We are indebted to the \href{https://en.wikipedia.org/wiki/Wikipedia_community}{Wikipedia community} for summarising important details of the sports competition discussed in the paper. \\
This research was funded by the Ministry of Culture and Innovation and the National Research, Development and Innovation Office under Grant Nr.~TKP2021-NVA-02. The work of \emph{Roland Molontay} is supported by the European Union project RRF-2.3.1-21-2022-00004 within the framework of the Artificial Intelligence National Laboratory.

\bibliographystyle{apalike}
\bibliography{All_references}

\end{document}